\documentclass[final]{siamltex}
 \usepackage{amsfonts}
 \usepackage{graphicx}
\usepackage{algorithm}
\usepackage{algorithmic}
 \usepackage{amsmath}
 \usepackage{mathrsfs}
 \usepackage{float}
 \usepackage{caption}
 \usepackage{epsfig,amsmath,amssymb}
\usepackage{array}
\usepackage{color}

\title{An Approach to Making SPAI and PSAI Preconditioning Effective
for Large Irregular Sparse Linear Systems\thanks{Supported
by National Basic Research Program of China 2011CB302400 and the
National Science Foundation of China (No. 11071140).}}

\author{Zhongxiao Jia\thanks{Department of Mathematical Sciences,
Tsinghua University, Beijing 100084, People's Republic of China
(jiazx@tsinghua.edu.cn).} \and Qian Zhang\thanks{Department of Mathematical
Sciences, Tsinghua University, Beijing 100084, People's Republic of China
(qianzhang.thu@gmail.com)}.}

\begin{document}

\maketitle
\floatname{algorithm}{Procedure}
\begin{abstract}
We investigate the SPAI and PSAI preconditioning procedures and
shed light on two important features of them: (i) For the large linear system
$Ax=b$ with $A$ irregular sparse, i.e., with $A$ having $s$ relatively
dense columns, SPAI may be very costly to implement, and the resulting sparse
approximate inverses may be ineffective for preconditioning. PSAI can be
effective for preconditioning but may require excessive storage and be unacceptably
time consuming; (ii) the situation is improved drastically when $A$ is regular
sparse, that is, all of its columns are sparse. In this case, both SPAI and PSAI
are efficient. Moreover, SPAI and, especially, PSAI are more likely to construct
effective preconditioners. Motivated by these features, we propose an approach
to making SPAI and PSAI more practical for $Ax=b$ with $A$ irregular sparse.
We first split $A$ into a regular sparse $\tilde A$ and a matrix of low rank $s$.
Then exploiting the Sherman--Morrison--Woodbury formula, we transform $Ax=b$
into $s+1$ new linear systems with the same coefficient matrix $\tilde A$,
use SPAI and PSAI to compute sparse approximate inverses of $\tilde A$
efficiently and apply Krylov iterative methods to solve the preconditioned
linear systems. Theoretically, we consider the non-singularity
and conditioning of $\tilde A$ obtained from some important classes
of matrices. We show how to recover an approximate solution of $Ax=b$ from
those of the $s+1$ new systems and how to design reliable stopping criteria
for the $s+1$ systems to guarantee that the approximate solution of $Ax=b$
satisfies a desired accuracy. Given the fact that irregular sparse linear
systems are common in applications, this approach widely extends the
practicability of SPAI and PSAI. Numerical results demonstrate
the considerable superiority of our approach to the direct application of SPAI
and PSAI to $Ax=b$.
\end{abstract}


\begin{AM}
65F10
\end{AM}
\pagestyle{myheadings}
\thispagestyle{plain}
\markboth{Z. JIA AND Q. ZHANG}{Making SPAI and PSAI Effective for Irregular
Sparse Systems}
\section{Introduction}\label{intro}

%======== 1. Problems, difficulty, and contribution ==========%
Krylov iterative solvers \cite{freund1992,saad2003} have been
very popular for the large sparse linear system
\begin{equation}\label{eq:Axb}
 Ax=b,
\end{equation}
where $A$ is a real nonsingular $n\times n$ matrix and $b$
is an $n$-dimensional real vector.  However, when $A$ has bad spectral property
or is ill conditioned, the solvers generally exhibit extremely slow
convergence and necessitate preconditioning techniques.
Sparse approximate inverse (SAI) preconditioning
aims to compute a preconditioner $M\approx A^{-1}$ directly so as to
improve the conditioning of (\ref{eq:Axb}) for the vast majority of problems,
and it is nowadays one class of important general-purpose preconditioning
techniques for Krylov solvers  \cite{benzi02,saad2003}. There are two typical
kinds of SAI preconditioning approaches. One of them constructs a factorized
sparse approximate inverse. An effective algorithm of this kind is the
approximate inverse (AINV) algorithm, which is derived from the
incomplete (bi)conjugation procedure \cite{benzi1996sparse,benzi98}.
The other kind is based on F-norm minimization
and is inherently parallelizable. This kind of preconditioners
are more robust and general. The approach constructs $M\approx A^{-1}$
by minimizing $\|AM-I\|_F$ for a specified pattern
of $M$ that is either prescribed in advance or determined adaptively,
where $\|\cdot\|_F$ denotes the F-norm of a matrix.
The most popular F-norm minimization-based SAI preconditioning technique may be
the adaptive sparse approximate inverse (SPAI) procedure~\cite{Grote1997},
which has been widely used. The adaptive power sparse
approximate inverse (PSAI) procedure with dropping,
advanced in \cite{Jia2009power} and called PSAI($tol$), is also an
effective F-norm minimization-based SAI preconditioning technique
and has been shown to be at least competitive with SPAI numerically
and can outperform SPAI for some practical problems. A hybrid version, i.e.,
the factorized approximate inverse (FSAI) preconditioning
based on F-norm minimization, has been introduced in \cite{KOLOTILINA1993}.
FSAI is generalized to block form, called BFSAI in \cite{janna2010block}.
An adaptive algorithm in \cite{janna2011adaptive}
is presented that generates automatically the nonzero pattern of a BFSAI preconditioner.
In addition, the idea of F-norm minimization is generalized in \cite{holland2005} by
introducing a sparse readily inverted {\em target} matrix $T$. $M$ is then computed
by minimizing $\|AM-T\|_{F,H}$ over a space of matrices with a prescribed sparsity
pattern, where $\|\cdot\|_{F,H}$ is the generalized F-norm defined
by $\|A\|^2_{F,H}=\langle A,A\rangle_{F,H}=trace(A^THA)$ with $H$ being some
symmetric positive definite matrix, the superscript $T$ denotes the
transpose of a matrix or vector. A good comparison of factorized SAI and
F-norm minimization based SAI preconditioning approaches can be
found in \cite{benzi1999comparative}. Spare approximate inverses have been
shown to provide effective smoothers for multigrid; see, e.g.,
\cite{Broker200261,broker:1396,sedlacek,Tang2000}.
For a comprehensive survey on preconditioning techniques,
we refer the reader to \cite{benzi02,ferronato}.

%======== 2. describe SPAI and PSAI($tol$) ==========%

Throughout this paper, we will frequently use two keywords
``regular sparse" and ``irregular sparse" for a matrix. By regular sparse,
we, qualitatively and sensibly, mean that all the columns of the matrix
are sparse, and no column has much more nonzero entries
than the others. By irregular sparse, we mean that there
are some (relatively) dense columns, each of which has considerably more
nonzero entries than the other sparse columns.
We call such columns irregular and denote by $s$ the number of them.
Problem (\ref{eq:Axb}) with $A$ irregular sparse is quite common and
arises in semiconductor device problem,
power network problem, circuit simulation, optimization problem
and many others \cite{davis2011university}.
Quantitatively, we will declare a matrix irregular sparse
if it has at least one column that has $10p$ nonzero entries or more,
where $p$ is the average number of nonzero entries per column.
Under this definition, we investigate all the matrices in
the University of Florida sparse matrix collection \cite{davis2011university},
which contains 2649 matrices and of them 1978 are square.
We find that 682 out of these 1978 matrices are irregular sparse.
That is, 34\% of the matrices in the collection are irregular sparse.
In the collection, there are some social networks, citation networks,
and other graphs that are not typically viewed as linear systems.
They often have dense columns. But even if these matrices are removed,
there are 30\% irregular sparse  matrices, 464 out of 1554,
in the collection. So irregular sparse linear systems have a wide broad of
applications.\footnote{We thank Professor Davis,
one of the authors of \cite{davis2011university}, very much for providing us such a
very valuable data analysis, clearly showing that irregular sparse linear problems
are quite common.}

The possible success of any SAI preconditioning procedure is based on
the crucial assumption that $A$ has good sparse approximate inverses.
Under this assumption, throughout the paper we consider the case
that $A$ is irregular sparse. It is empirically observed
that good sparse approximate inverses of $A$ are irregular sparse too.
In the context, we are concerned with the adaptive SPAI and PSAI($tol$) procedures.
As is seen, since the number of nonzero entries in an individual column of
the final $M$ in SPAI is bounded by the maximum number of most profitable indices
per loop times the maximum loops, a column of $M$ may have not enough
nonzero entries. As a result, $M$ obtained by SPAI may not
approximate $A^{-1}$ well and thus may be ineffective for preconditioning.
Moreover, it can be justified that the SPAI algorithm
may be very costly to implement.
For PSAI($tol$), we may also suffer the unaffordable overhead from solving
some possibly large LS problems (\ref{lsprob}), although it is more likely to
construct effective preconditioners no matter whether $A$ is regular sparse or not.
Remarkably, it turns out that the situation mentioned above is improved
substantially if $A$ is regular sparse. With reasonable parameters,
SPAI and PSAI($tol$) are efficient. Furthermore, SPAI and, especially,
PSAI($tol$) are more likely to construct effective sparse approximate inverses.

It is well known \cite{benzi2000orderings} that the computational consumption,
stability and effectiveness of factorized SAI preconditioners are generally
sensitive to reorderings of $A$. Unfortunately,
reorderings do not help for SPAI and PSAI($tol$). The reason is that reorderings do
not change the irregularity of sparsity patterns of $A$, $A^{-1}$
and good sparse approximate inverses of $A$. Therefore, with reorderings used, SPAI
and PSAI($tol$) may still be very costly to implement, and SPAI may still be
ineffective for preconditioning. We refer the reader to \cite{benzi02} for the
relevant arguments about SPAI, which are valid for PSAI($tol$) as well.

Because of the above features, we naturally come up with the idea of
transforming the irregular sparse problem (\ref{eq:Axb})
into some regular sparse one(s), on which SPAI and PSAI($tol$) may work well.
It will appear that the Sherman--Morrison--Woodbury formula \cite{golub,Stewart1998}
provides us a powerful tool and can be used a key step towards our goal.
We present an approach to splitting $A$ into a regular sparse matrix $\tilde A$ and a
matrix of low rank $s$ and to transforming (\ref{eq:Axb}) into the $s+1$ new linear
systems with the same coefficient matrix $\tilde A$. By exploiting the
Sherman--Morrison--Woodbury formula, we can recover the solution of (\ref{eq:Axb})
from the ones of the $s+1$ systems directly.
We consider numerous practical issues on how to obtain a desired splitting of $A$,
how to define and compute an approximate solution of (\ref{eq:Axb}) via those of
the new systems, and how accurately we should solve the new systems, etc. A
remarkable merit of this approach is that SPAI and PSAI($tol$) are efficient to
construct possibly effective preconditioners for the new systems, making Krylov
solvers converge fast. The price we pay is to solve $s+1$ linear systems.
But a great bonus is that we only need to construct one effective sparse
approximate inverse $M$ efficiently for the $s+1$ systems.
The price is generally insignificant as it is typical that the construction
of an effective $M$ dominates the whole cost of Krylov iterations even
in a parallel computing environment \cite{barnard1999,benzi02,benzi1999comparative}.
As a matter of fact, due to inherent parallelizations
of SPAI and PSAI($tol$), SAI type precondtioners
are attractive for solving a sequence of linear systems with the
same coefficient matrix, as has been addressed in the literature, e.g.,
\cite{benzi98}. Therefore, given the fact that irregular sparse
linear systems are quite common in applications, our approach widely extends
the practicality of SPAI and PSAI($tol$).

%======== 4. outline of this paper ==========%
The paper is organized as follows. In \S\ref{overview}, we review
SPAI and PSAI($tol$) procedures and shed light on the above-mentioned features for
$A$ irregular sparse and regular sparse, respectively. In \S\ref{approach}, we
describe our new approach for solving (\ref{eq:Axb}) with $A$ irregular sparse.
In \S\ref{issue} we consider numerous theoretical and practical issues,
establishing some results on the nonsingularity of $\tilde A$ obtained from certain
important and widely useful classes of matrices $A$ and drawing some claims
on the conditioning of $\tilde A$.
In \S\ref{numerexp}, we report numerical experiments to demonstrate the
superiority of our approach to SPAI and PSAI($tol$)
applied to (\ref{eq:Axb}) directly and the superiority of
PSAI($tol$) to SPAI for both irregular and regular sparse linear systems.
Finally, we conclude the paper in \S\ref{conclude}.

\section{The SPAI and PSAI($tol$) procedures}\label{overview}

In this section, we overview the SPAI and PSAI($tol$) procedures and
shed light on the facts that (i) SPAI and PSAI($tol$) may be very costly
and (ii) SPAI may not be effective for preconditioning if $A$ is irregular
sparse.

During loops each of SPAI and PSAI($tol$) solves a sequence of constrained
optimization problems of the form
\begin{equation}\label{miniz}
\min_{M\in\mathcal{M}}\|AM-I\|_F,
\end{equation}
where $\mathcal{M}$ is the set of matrices with a given sparsity pattern
$\mathcal{S}$. Denote by $\mathcal{M}_k$ the set of $n$-dimensional vectors
whose sparsity pattern is $\mathcal{S}_k
=\{i|(i,k)\in\mathcal{S}\}$. Then (\ref{miniz}) is decoupled into $n$
independent constrained least squares (LS) problems
\begin{equation}\label{subprob}
 \min_{m_k\in\mathcal{M}_k}\|Am_k-e_k\|, \quad k=1,2,\ldots,n
\end{equation}
with $e_k$ the $k$-th column of the $n\times n$ identity matrix $I$. Here and
hereafter, the norm $\|\cdot\|$ denotes the vector 2-norm or the matrix spectral
norm. For each $k$, let $\hat m_k=m_k(\mathcal S_k)$, ${\cal L}_k$ be the set
of indices of nonzero rows of $A(:,{\cal S}_k)$, $\hat A_k=
A({\cal L}_k,{\cal S}_k)$ and $\hat e_k=e_k({\cal L}_k)$.
Then (\ref{subprob}) is reduced to the smaller unconstrained LS problems
\begin{equation}\label{lsprob}
 \min_{\hat m_k}\|\hat A_k\hat m_k-\hat e_k\|, \quad k=1,2,\ldots,n,
\end{equation}
which can be solved by the QR decomposition in parallel.
SPAI and PSAI($tol$) determine the sparsity pattern
$\mathcal{S}_k$ dynamically: starting with a simple initial pattern,
say the pattern of $e_k$, $\mathcal{S}_k$ is augmented or adjusted adaptively
until the residual norm $\|Am_k-e_k\|$ falls below a given tolerance or
the maximum number of augmentations is reached.
The distinction between SPAI and PSAI($tol$) lies in the way that $\mathcal S_k$
is augmented or adjusted. As is clear from \S\ref{spai} and \S\ref{psai},
the already existing positions of nonzero entries in $m_k$
for SPAI are retained in subsequent loops, and at each loop a few
most profitable indices added, which are selected from a certain new set generated
at the current loop. PSAI($tol$) aims to adaptively drop the entries
whose sizes are below certain tolerances and only retain the remaining large
ones during the determination of $M$, and $\mathcal S_k$ is adjusted dynamically
by not only absorbing new members but also discarding the positions
where the entries of $m_k$ become small during loops.
In other words, at each loop PSAI($tol$) determines the positions of
entries of large magnitude in a globally optimal sense,
while SPAI does the job locally by adding a few ones from a local pattern generated
at the current loop. Therefore, PSAI($tol$) may capture a more effective
sparsity pattern of $A^{-1}$ than SPAI.

\subsection{The SPAI procedure}\label{spai}

Denote by $\mathcal S_k^{(l)}$ the sparsity pattern of $m_k$ after $l$
loops of augmentation starting with a given initial pattern $\mathcal S_k^{(0)}$,
and by $\mathcal L_k^{(l)}$ the set of indices of nonzero rows of
$A(:,\mathcal S_k^{(l)})$. Let $\hat A_k=A(\mathcal L_k^{(l)},\mathcal S_k^{(l)}),
\hat e_k=e_k(\mathcal L_k^{(l)})$, and $\hat m_k$ be the solution of (\ref{lsprob}).
Then the residual of (\ref{subprob}) is
$$
r_k = A(:,\mathcal S_k^{(l)})\hat m_k-e_k.
$$
For $r_k\neq 0$, denote by $\mathcal L$ the set of indices $l$ for which $r_k(l)\neq 0$,
and by $\mathcal N$ the set of indices of nonzero columns of $A(\mathcal L, :)$.
Then
\begin{equation}\label{candi}
\tilde{\mathcal J} = \mathcal N\backslash\mathcal S_k^{(l)}
\end{equation}
constitutes the new candidates for augmenting $\mathcal S_k^{(l)}$ in the next loop.
Grote and Huckle~\cite{Grote1997} suggest to select several most profitable indices
from $\tilde{\mathcal J}$ and augment them to $\mathcal S_k^{(l)}$ to obtain
a new sparsity pattern $\mathcal S_k^{(l+1)}$ of $m_k$. They do this as follows:
for each $j\in\tilde{\mathcal J}$, consider the one-dimensional minimization problem
\begin{equation}\label{ls1d}
\min_{\mu_j}\|r_k+\mu_jAe_j\|,
\end{equation}
whose solution is
\begin{equation}\label{muj}
 \mu_j=-\frac{r_k^TAe_j}{\|Ae_j\|^2}.
\end{equation}
The 2-norm $\rho_j$ of the new residual $r_k+\mu_jAe_j$ satisfies
\begin{equation}\label{rhoj}
 \rho_j^2=\|r_k\|^2-\frac{(r_k^TAe_j)^2}{\|Ae_j\|^2.}
\end{equation}
The set $\tilde{\mathcal S}_k^{(l)}$ of the most profitable indices $j$
consists of those associated with a few, say, 1 to 5, smallest $\rho_j$
and is added to $\mathcal S_k^{(l)}$ to obtain $\mathcal S_k^{(l+1)}$.
Update $\mathcal L_k^{(l)}$ to get $\mathcal L_k^{(l+1)}$
by adding the set $\tilde{\mathcal L}_k^{(l)}$ of indices of new
nonzero rows corresponding to $\tilde{S}_k^{(l)}$.
The new augmented LS problem (\ref{lsprob}) is solved by updating $\hat m_k$
instead of resolving it. Proceed in such a way until
$\|Am_k-e_k\|\leq\delta$ or $l$ attains the prescribed maximum
$l_{\max}$ of loops, where $\delta$ is a given mildly small tolerance, say
$0.1\sim 0.4$.

{\em Remark 1.}
Let us consider the computational complexity of SPAI. For $A$ regular sparse,
it is straightforward to verify that $r_k=Am_k-e_k$ is also sparse, and
both $\mathcal L$ and $\tilde{\mathcal J}$ have only a few elements.
As a result, the cardinal number of $\mathcal S_k^{(l)}$ is small,
so is the order of $\hat A_k$ in (\ref{lsprob}). Therefore, it is cheap to
determine the set $\tilde{\mathcal S}_k^{(l)}$ of the most profitable
indices and solve (\ref{lsprob}).
However, the situation deteriorates severely when $A$ is irregular sparse.
For example, assume that the $k$-th column $a_k$ of $A$ is relatively
dense, and denote $A=(a_{ij})$. If $a_{kk}\neq 0$ and we take
$\mathcal S_k^{(0)}=\{k\}$, the residual $r_k=m_k(k)a_k-e_k$ is as dense as $a_k$
in the first loop of SPAI, causing that $\mathcal L$, ${\mathcal N}$
and $\tilde{\mathcal J}$ have big cardinal numbers. Keep in mind that
$\mathcal S_k^{(l)}$ has a very small cardinal number for all $l$
as both the maximum of loops and the number of most
profitable indices are small.
Then it is easily checked that ${\mathcal L}$, ${\mathcal N}$
and $\tilde{\mathcal J}$ always have very big cardinal numbers
in subsequent loops. As a consequence, suppose that $a_k$ is fully dense,
at each loop we have to compute almost $n$ numbers $\rho_j$, order them and
select the most profitable indices in $\tilde{\mathcal J}$. Generally, at some
loop, once a nonzero index of $m_k$ corresponds to an irregular column of $A$,
then the resulting residual $r_k$ must be dense, generating big cardinalities
of ${\mathcal L},\ {\mathcal N}$ and $\tilde{\mathcal J}$ at the current
loop and in subsequent loops. So SPAI may be very costly to implement for $A$
irregular sparse.
\smallskip

{\em Remark 2}. SPAI provides a right preconditioner. One should notice that $A^T$ may
not be so when $A$ is irregular sparse. Naturally, one might apply SPAI to $A^T$
and computes a left preconditioner $M$, whose transpose $M^T$ is a right
preconditioner. However, SPAI may still be very costly to
implement in this way, and in fact it may be more costly: Suppose
that the $k$-th row $a_k^T$ of $A^T$ is fully dense. Then when computing the $j$-th
column $m_j$ of $M$, $j=1,2,\ldots,n$, since $a_k^T$ is dense, we generally have
$a_k^Tm_j\not=0$. This means that the $k$-th component of
the residual $r_j=A^Tm_j-e_j$ is nonzero and thus $k\in{\mathcal L}$ at each loop.
As a result, when computing each column $m_j$ of $M$, the cardinalities of
$\mathcal N$ and $\tilde{\mathcal J}$ are about $n$
at each loop, and we have to compute almost $n$ numbers $\rho$, order them and
select a few most profitable indices at each loop.  Such kind of feature is the
same for all $j=1,2,\ldots,n$. Consequently, the situation is now more severe than SPAI
working on $A$ directly, and generally it is more costly to apply SPAI to $A^T$
when $A$ is irregular sparse.
\smallskip

{\em Remark 3}. For the irregular sparse $A$, suppose that
it has good sparse approximate inverses. Then they are typically
irregular sparse too. Suppose the $k$-th column of a good sparse approximate
inverse is irregular. The description of SPAI shows that if $a_k$ is irregular,
i.e., relatively dense, then the sets
$\tilde{\mathcal J}$ in (\ref{candi}) have big cardinal numbers during loops.
Nevertheless, SPAI simply takes the set $\tilde{\mathcal S}_k^{(l)}$ to be
{\em only a few} most profitable indices from $\tilde{\mathcal J}$
at each loop. If the cardinality of $\tilde{\mathcal S}_k^{(l)}$ is fixed small,
say 5, the default value as suggested and used in \cite{barnard1999,barnard,Grote1997},
then the $k$-th column of $M$ is sparse and may not approximate the
$k$-th column of $A^{-1}$ well unless the loops $l_{\max}$ is
large enough. Since the number of nonzero entries in an individual column of
the final $M$ in SPAI is bounded by the maximum number of most profitable
indices per loop times the maximum loops $l_{\max}$, which is fairly
small, say 20 (the default value is 5 in \cite{barnard1999,barnard}),
a column of $M$ may have not enough nonzero entries.
As a result, $M$ obtained by SPAI may not approximate $A^{-1}$ well
and is thus ineffective for preconditioning.
\smallskip

{\em Remark 4}. We point out that there are pathological regular sparse
matrices whose good sparse approximate inverses are irregular sparse.
Thus it is possible, and in fact quite common in practice, that a regular
sparse matrix causes problems for SPAI. This is a case for sparse
matrices arising from finite differences, volumes or elements of some
PDEs. A two-level sparse approximate inverse
preconditioning proposed by Chen \cite{chen2006} attempts to handle
this class of problems, in which SPAI is used to first compute a right
preconditioner $M_1$ of $A$ and then compute a left preconditioner
$M_2$ of the sparsification of $AM_1$. Numerically, this procedure
can be effective for preconditioning a number of problems,
but it lacks theoretical justification and may encounter difficulty since
the sparsification of $AM_1$ is crucial but it can only be done empirically.

\subsection{The PSAI($tol$) procedure}\label{psai}

We first review the basic PSAI (BPSAI) procedure \cite{Jia2009power}. From the
Cayley--Hamilton theorem, $A^{-1}$ can be expressed as a matrix polynomial of $A$
of degree $d-1$ with $d\leq n$:
\begin{equation}
 A^{-1}=\sum_{i=0}^{d-1}c_iA^i，
\end{equation}
with $A^0=I$ and $c_i$ $(i=0,1,\ldots,d-1)$ being certain constants.
Denote $\mathcal P(\cdot)$ by the sparsity pattern of a matrix or vector,
and write the matrix $|A|=(|a_{ij}|)$.
It is obvious that $\mathcal P(A^{-1})\subseteq\mathcal P((I+|A|)^{d-1})$.
For a given small positive integer $l_{\max}<d$, the pattern
$\mathcal{S}$ of a sparse approximate inverse $M$ is taken as a subset of
$\mathcal P((I+|A|)^{l_{\max}})$ in BPSAI.
Therefore, the pattern $\mathcal S_k$ of the $k$-th column $m_k$ of $M$
is a subset of $\bigcup\limits_{l=0}^{l_{\max}} \mathcal{P}(|A|^le_k)$ since
$$
\mathcal P((I+|A|)^{l_{\max}})=\bigcup_{l=0}^{l_{\max}}\mathcal{P}(|A|^l).
$$

The adjustment of $\mathcal S_k$ for $1 \le k \le n$ proceeds dynamically
as follows. For $l=0,1,\ldots,l_{\max}$, denote by $\mathcal S_k^{(l)}$
the sparsity pattern of $m_k$ at the current loop $l$ and by $\mathcal L_k^{(l)}$
the set of indices of nonzero rows of $A(:,\mathcal S_k^{(l)})$. Set
$a_k^{(l)}=A^le_k$. Then  $a_k^{(l+1)}=Aa_k^{(l)}$ with $a_k^{(0)}=e_k$.
The sparsity pattern $\mathcal S_k^{(l+1)}$ is updated as
$\mathcal S_k^{(l+1)}=\mathcal P(a_k^{(l+1)})\cup\mathcal S_k^{(l)}$.
In the next loop we solve the augmented LS problem
$$
\min\|A(\mathcal L_k^{(l+1)},\mathcal S_k^{(l+1)})m_k(\mathcal S_k^{(l+1)})-
e_k(\mathcal L_k^{(l+1)})\|
$$
with $\mathcal L_k^{(l+1)}=\mathcal L_k^{(l)}\cup \tilde{\mathcal L}_k^{(l)}$,
where $\tilde{\mathcal L}_k^{(l)}$ is the set of indices of new
nonzero rows corresponding to the set $\mathcal P(a_k^{(l+1)})
\backslash\mathcal S_k^{(l)}$.
This problem corresponds to the small LS problem (\ref{lsprob}),
whose solution can be updated from $m_k(\mathcal S_k^{(l)})$ efficiently.
Proceed in such a way until $\|Am_k-e_k\|\leq\delta$ or $l>l_{\max}$.

It has been proved in \cite[Theorem 1]{Jia2009power} that for
$\mathcal S_k^{(0)}=\{k\}$, as $l_{\max}$ increases, $M$ obtained by
BPSAI may become increasingly denser
quickly once one column in $A$ is irregular sparse. In order to control the
sparsity of $M$ and construct an effective preconditioner, some reasonable
dropping strategies should be used. Two practical PSAI($l\!f\!ill$)
and PSAI($tol$) algorithms have been proposed in \cite{Jia2009power}.
PSAI($l\!f\!ill$) aims to retain at most $l\!f\!ill$ entries of large
magnitude in $m_k$. To be more flexible, $l\!f\!ill$ may vary with $k$.
Since an effective sparsity pattern of $A^{-1}$ is generally
unknown in advance, it is difficult to prescribe a reasonable $l\!f\!ill$.
This is a shortcoming similar to SPAI. In contrast, for the newly computed
$m_k$ at loop $l$, PSAI($tol$) drops those entries of small magnitude below
certain tolerances $tol$ and retains only the ones of large magnitude.
Therefore, PSAI($tol$) is more reasonable and reliable to capture an
effective sparsity pattern of $A^{-1}$ and determines
the corresponding entries. A central issue is the selection
of dropping tolerance $tol$. This issue is mathematically nontrivial,
and $tol$ has strong effects on the effectiveness of PSAI($tol$) and many other SAI
preconditioning procedures. The authors \cite{Jia2012robust}
have proposed effective and robust dropping criteria for
PSAI($tol$) and all the static F-norm minimization-based SAI preconditioning
procedures. For PSAI($tol$), it is shown that, at loop $l\leq l_{\max}$,
a nonzero entry $m_{jk}$ is dropped for $1\leq j\leq n$ if
\begin{equation}\label{tolk}
|m_{jk}|\leq tol_k=\frac{\delta}{nnz(m_k)\|A\|_1},\ k=1,2,\ldots,n,
\end{equation}
where $nnz(\cdot)$ is the number of nonzero entries in a vector or matrix,
and $\delta<0.5$ is the stopping tolerance for SPAI and BPSAI.
We comment that $m_k$ in (\ref{tolk}) is the
newly computed one at loop $l$. This criterion makes $M$ as sparse
as possible and meanwhile has similar
preconditioning quality to the possibly much denser one
obtained by BPSAI \cite{Jia2012robust}. More precisely,
for the final $M$ obtained by PSAI($tol$), if (\ref{tolk}) is used then
the residual norm $\|Am_k-e_k\|\leq2\delta,\ k=1,2,\ldots,n$ when
the residual norm of each column of the preconditioner obtained by
BPSAI falls below $\delta$.
\smallskip

{\em Remark.} If $A$ is regular sparse, it is direct to justify
that the size of $\hat A_k$ in (\ref{lsprob}) is small
for a small $l_{\max}$ and it is cheap to solve (\ref{lsprob}).
However, the situation changes sharply for $A$ irregular sparse.
Suppose that the $k$-th column $a_k$ of $A$ is relatively dense and
and take $\mathcal S_k^{(0)}=\{k\}$.
Then the resulting $m_k$ is also relatively dense since
$\mathcal S_k^{(1)}=\mathcal P(a_k)\cup\{k\}=\mathcal P(a_k)$. Therefore,
(\ref{lsprob}) is a relatively large LS problem. Since 
$\mathcal S_k^{(l+1)}=\mathcal P(a_k^{(l+1)})\cup\mathcal S_k^{(l)}$, 
(\ref{lsprob}) is always a large LS problem at each subsequent 
loop $l\leq l_{\max}$. Therefore, if $a_k$ is dense,
PSAI($tol$) is very costly and impractical.

\section{Transformation of (\ref{eq:Axb}) into regular sparse linear systems}
\label{approach}

As we have seen, SPAI and PSAI($tol$) may be very costly to implement for $A$
irregular sparse, and $M$ obtained by SPAI may be ineffective for preconditioning
(\ref{eq:Axb}). In this section, we attempt to transform (\ref{eq:Axb}) into
some regular sparse ones, making SPAI and PSAI($tol$) more practical to construct
possibly effective $M$. It turns out that the following Sherman--Morrison--Woodbury
formula (see \cite[p. 50]{golub} and \cite[p. 330]{Stewart1998})
provides us a powerful tool for our purpose.

\begin{theorem}\label{thm1}
Let $U, V\in \mathbb{R}^{n\times s}$ with $s\leq n$.
If $A$ is nonsingular, then $A-UV^T$ is nonsingular if and only if
$I-V^TA^{-1}U$ is nonsingular. Furthermore,
\begin{equation}\label{Sherman--Morrison--Woodbury1}
 (A-UV^T)^{-1} = A^{-1}+A^{-1}U(I-V^TA^{-1}U)^{-1}V^TA^{-1}.
\end{equation}
\end{theorem}

The formula is typically of interest for $s\ll n$, and it is called the
Sherman--Morrison formula when $s=1$. For a good survey on the history and
applications, we refer the reader to \cite{hager1989}.

For our purpose, assume that the $j_1,j_2,\ldots,j_s$-th columns of
$A$ are irregular and the remaining $n-s$ ones are sparse. Denote
by $A_{dc}=(a_{j_1},a_{j_2},\ldots,a_{j_s})$ the matrix consisting of the $s$
irregular columns of $A$, and by
$\tilde{A}_{dc}=(\tilde{a}_{j_1},\tilde{a}_{j_2},\ldots,\tilde{a}_{j_s})$
the sparsification of $A_{dc}$ that drops some of its nonzero entries,
so that each column of $\tilde{A}_{dc}$ is as sparse as the other $n-s$ columns
of $A$. Define $U = A_{dc} - \tilde{A}_{dc}=(u_1,u_2,\ldots,u_2)$.
Then the nonzero entries of $U$ are just those dropped ones of $A_{dc}$.
Let $\tilde A$ be the matrix that is obtained from $A$ by replacing its dense
columns $a_{j_i}$ by the sparse vectors $\tilde{a}_{j_i}$, $i=1,2,\ldots,s$.
Then $\tilde A$ is regular sparse and satisfies
\begin{equation}\label{split}
A=\tilde A+UV^T,
\end{equation}
where
$V = (e_{j_1},e_{j_2},\ldots,e_{j_s})$ with $e_{j_i}$ the $j_i$-th column of
the $n\times n$ identity matrix $I ,\ i=1,2,\ldots,s$.
Assume that $\tilde A$ is nonsingular. Then it follows from
(\ref{Sherman--Morrison--Woodbury1}) that
\begin{equation}\label{solution}
A^{-1}=\tilde A^{-1}-\tilde A^{-1}U(I+V^T\tilde
A^{-1}U)^{-1}V^T\tilde A^{-1}.
\end{equation}
Therefore, the solution of (\ref{eq:Axb}) is
\begin{equation}\label{update}
x=A^{-1}b=\tilde A^{-1}b-(\tilde A^{-1}U)(I+V^T(\tilde A^{-1}U))^{-1}
(V^T\tilde A^{-1}b).
\end{equation}
This amounts to solving a new regular sparse linear system
\begin{equation}\label{regprob}
\tilde A y=b
\end{equation}
and the other $s$ regular sparse linear systems
\begin{equation}\label{auxi}
\tilde{A} w_j=u_j, \ j=1,2,\ldots,s.
\end{equation}
If the exact solutions to (\ref{regprob}) and (\ref{auxi}) were
available, we would get the solution $x$ of (\ref{eq:Axb})
from (\ref{update}) by solving the small $s\times s$ linear system with
the coefficient matrix $I+V^T(\tilde A^{-1}U)$ and the right-hand side
$V^T\tilde A^{-1}b$.

We can summarize the above approach as Procedure *.

\begin{algorithm}[H]
\caption*{\textbf{Procedure *: Solving the irregular sparse linear system
(\ref{eq:Axb})}}
\begin{algorithmic}[1]
\STATE Find $s$ and $A_{dc}$, and sparsify $A_{dc}$ to get $\tilde{A}_{dc}$.
Define $U=A_{dc}-\tilde{A}_{dc}$ and the regular sparse matrix
$\tilde A=A-UV^T$, where $V = (e_{j_1},e_{j_2},\ldots,e_{j_s})$.

\STATE Solve $s+1$ linear systems (\ref{regprob}) and (\ref{auxi}) for $y$ and
$w_1,w_2,\ldots,w_s$, respectively.

\STATE Let $W=(w_1,w_2,\ldots,w_s)=\tilde A^{-1}U$ and compute
the solution $x$ of $Ax=b$ by
\begin{equation}\label{update1}
x=A^{-1}b=y-W(I+V^TW)^{-1}(V^Ty).
\end{equation}
\end{algorithmic}
\end{algorithm}

For regular sparse systems (\ref{regprob}) and (\ref{auxi}), we suppose that
only iterative solvers are viable in our context. Now, a big and direct reward is
that SPAI and PSAI($tol$) can be implemented much more efficiently
to construct preconditioners for the $s+1$ regular sparse systems with the same
regular sparse coefficient matrix $\tilde A$. Furthermore, compared with the
irregular sparse case, SPAI is more likely to construct
an effective preconditioner now.

\section{Theoretical and practical considerations}\label{issue}

When iterative solvers are used, recovering an \textit{approximate} solution of
(\ref{eq:Axb}) via those of (\ref{regprob}) and (\ref{auxi}) is quite involved
and is not as simple as Procedure * indicates, in which the exact $y$ and $W$
are assumed. In order to use Procedure * to develop a practical iterative solver
for (\ref{eq:Axb}), we first need to handle several theoretical and practical issues.

The first issue is about the quantitative meaning of irregular columns, by which
we define $A_{dc}$. Obviously, like sparsity itself and many other
quantities in numerical analysis, it appears impossible to give a
precise definition of it. In fact, it is also unnecessary to do so. In our
experiments, we empirically find that the threshold $10p$ is a good choice,
where $p=\lfloor nnz(A)/n\rfloor$ is the average number of nonzero entries
per column of $A$. If the number of nonzero entries in a column exceeds $10p$,
then we mark it as an irregular column. Based on this criterion,
we determine all the irregular columns of $A$ and the number $s$ of them.
Numerically, we have found that other thresholds ranging from $8p$ to $15p$
work well and exhibit no essential difference. So our approach is insensitive to
thresholds.

The second issue is which nonzero entries in $A_{dc}$ should be dropped to
get $\tilde{A}_{dc}$ and generate $\tilde A$. In principle, the number
$\hat p$ of nonzero entries in each column of $\tilde{A}_{dc}$ should be
comparable to $p$. For the choice of $\hat p$, to be unique, we simply propose
taking $\hat p=p$. Given such $\hat p$, there may be many dropping ways.
Two obvious approaches can be adopted. The first approach is to
retain the diagonal and the $p-1$ nonzero entries nearest to the
diagonal in each column of $A_{dc}$. The second approach is to retain the
diagonal and the other $p-1$ largest entries in magnitude of each column of
$A_{dc}$. Numerically, two approaches have exhibited very similar behavior.
Therefore, we will take the first approach and report the results obtained.

The third issue is on the non-singularity of $\tilde A$, which is
crucial both in theory and practice. First of all, we present the following results.

\begin{theorem}\label{thm3}
$\tilde A$ constructed above is nonsingular for the following classes of matrices:
\begin{romannum}
\item $A$ is strictly (row or column) diagonally dominant.

\item $A$ is irreducibly (row or column) diagonally dominant.

\item $A$ is an $M$-matrix.
\end{romannum}
\end{theorem}

\begin{proof}
(i). If $A$ is strictly (row or column) diagonally dominant,
$\tilde A$ is so too since it removes some off-diagonal nonzero
entries of $A$. Therefore, $\tilde A$ is nonsingular \cite[p. 23]{varga}.

(ii). For $A$ irreducibly (row or column) diagonally dominant, $\tilde A$ is
either irreducible or reducible. If $\tilde A$
is irreducible, then it must be irreducibly (row or column) diagonally
dominant. So $\tilde A$ is nonsingular \cite[p. 23]{varga}.

If $\tilde A$ is reducible, without loss of generality  we suppose that there
is a permutation matrix $P$ such that
$$
P\tilde A P^T=\left(\begin{array}{cc}
\tilde A_{11} & \tilde A_{12}\\
0 &\tilde A_{22}
\end{array}\right),
$$
where $\tilde A_{11}$ and $\tilde A_{22}$ are irreducibly square matrices.
Since $A$ is irreducibly (row or column) diagonally dominant, $PAP^T$
is so too. Partition
$$
PA P^T=\left(\begin{array}{cc}
A_{11} & A_{12}\\
A_{21} & A_{22}
\end{array}\right)
$$
conformingly. Then each of $A_{12}$ and $A_{21}$ must have nonzero entries;
otherwise, $A$ is reducible. So $A_{11}$ and $A_{22}$ must be (row or column)
diagonally dominant, and at least one row or column in each of them is
strictly diagonally dominant. Note that all the nonzero entries
of $P\tilde AP^T$ are the same as the corresponding ones of $PAP^T$.
Therefore, $\tilde A_{11}$ and $\tilde A_{22}$
are (row or column) diagonally dominant, and at least one row or column
in each of them is strictly diagonally dominant. So,
both $\tilde A_{11}$ and $\tilde A_{22}$ are (row or column) diagonally
dominant. Since $\tilde A_{11}$ and $\tilde A_{22}$ are irreducible,
both of them are nonsingular \cite[p. 23]{varga}, which means that
$P\tilde AP^T$ is nonsingular, so is $\tilde A$.

(iii). By the definition in \cite[p. 61]{varga}, an $M$-matrix assumes its non-singularity.
Theorem 3.25 of \cite[p. 91]{varga} states that any matrix $C$
obtained from the nonsingular $M$-matrix $A$ by setting certain off-diagonal
entries of $A$ to zero is also a nonsingular $M$-matrix. Since our
$\tilde A$ is just such a $C$, it is nonsingular.
\end{proof}

Three classes of matrices in the theorem have a wide broad of applications,
e.g., discretizations of second-order ODEs and elliptic PDEs, quantum chemistry,
information theory, stochastic process, systems theory, networks and modern
economics, to name only a few. Strictly row diagonally dominant matrices are
a class of special $H$-matrices; see \cite[Theorem 3.27, p. 92]{varga}.
Actually, the first two classes of matrices in the theorem can be extended to more
general forms, as the following corollary states.

\begin{corollary}\label{cor1}
$\tilde A$ is nonsingular for the following matrices:
\begin{romannum}
\item There exists a nonsingular
diagonal matrix $D$ for which $AD$ or $DA$ is strictly row or column
diagonally dominant, respectively.

\item There exist permutation matrices
$P$ and $Q$ for which $PAQ$ is strictly (row or column) diagonally
dominant.

\item There are a nonsingular diagonal matrix $D$ and permutation matrices
$P$ and $Q$ for which $(PAQ)D$ or $D(PAQ)$ is strictly row or column
diagonally dominant.

\item The irreducible analogues of the matrices in (i)--(iii).

\item $A$ is a H-matrix.
\end{romannum}
\end{corollary}

\begin{proof}
The proofs of Parts (i)--(iv) are direct from those of Theorem~\ref{thm3}. Assertion 5
holds because of a result of \cite[p. 124]{hornjohnson}, which states that a necessary
and sufficient condition for $A$ to be a $H$-matrix is that there
exists a nonsingular diagonal matrix $D$ for which $AD$ is strictly
row diagonally dominant.
\end{proof}

From the proof of (v) in Corollary~\ref{cor1}, we see that H-matrices
are a subclass of matrices in (i).

There should be more classes of matrices for which the resulting $\tilde A$
is nonsingular theoretically.
We do not pursue this topic further in this paper. For a general
real-world nonsingular $A$ even though it does not belong to
the classes of matrices in Theorem~\ref{thm3} and Corollary~\ref{cor1}.

The fourth issue is on the conditioning of $\tilde A$. For a general $A$,
it is possible to get either a well-conditioned or ill-conditioned
$\tilde A$. Given that $A$ is generally ill conditioned, the former is
more preferable, but we should not expect too much and the
latter is more possible. Theoretically speaking, $\tilde A$ may be worse
or better conditioned than $A$.  For a general irregular sparse $A$,
numerical experiments will indicate that $\tilde{A}$ is rarely worse
conditioned and in fact often, though not always, better conditioned than $A$.

However, for the first and third classes of matrices in Theorem~\ref{thm3},
we can analyze the conditioning of $\tilde A$ and show that $\tilde A$
may be generally better conditioned than $A$. More precisely,
for a strictly (row or column) diagonally dominant matrix $A$,
it is expected that the 1-norm condition number $\kappa_1(\tilde A)$ or
the infinity norm condition number $\kappa_{\infty}(\tilde A)$
is generally no more and may be considerably smaller than
$\kappa_1(A)$ or $\kappa_{\infty}(A)$. Similar claims hold for an $M$-matrix $A$.
Next we first look into the case that $A$ is
strictly row diagonally dominant. The case that $A$ is strictly column diagonally
dominant can be treated similarly.  Denote $\tilde A=(\tilde a_{ij})$, and define
the quantities
\begin{eqnarray*}
\beta_i&=&|a_{ii}|-\sum_{j\not=i}|a_{ij}|,\ i=1,2,\ldots,n,\\
\tilde\beta_i&=&|\tilde a_{ii}|-\sum_{j\not=i}|\tilde a_{ij}|,\ i=1,2,\ldots,n.
\end{eqnarray*}
Since that $\tilde a_{ii}=a_{ii}$ and all the nonzero entries $\tilde a_{ij}=a_{ij}$,
we have $\tilde\beta_i\geq \beta_i,\ i=1,2,\ldots,n$, with some strict
inequalities holding as $\tilde A$ drops some off-diagonal nonzero entries
in the $s$ irregular columns of $A$.
More precisely, from the definitions of $\tilde\beta_i$ and $\beta_i$,
it can be easily verified that if $i$ is in the set of the row indices of nonzero
entries in $U$ then $\tilde\beta>\beta_i$. The number of such $i$ is (much) bigger
than $s$ and can be very near to $n$ whenever $A$ has a fully dense column as
$n-p$ nonzero entries are dropped from an irregular column and put into a column of $U$.
Since $A$ is strictly row diagonally dominant, we have $\beta_i>0,\,i=1,2,\ldots,n$.
A result of \cite[p. 154]{higham} gives the following general bounds
\begin{eqnarray*}
\|A^{-1}\|_{\infty}&\leq&\frac{1}{\min_i\beta_i},\\
\|\tilde A^{-1}\|_{\infty}&\leq&\frac{1}{\min_i\tilde\beta_i}
\end{eqnarray*}
with $\|A^{-1}\|_{\infty}=1/\beta$ if all $\beta_i=\beta$ and
$\|{\tilde A}^{-1}\|_{\infty}=1/\tilde\beta$ if all $\tilde\beta_i=\tilde\beta$; see,
e.g., \cite{volkov}.
Since $\tilde\beta_i\geq\beta_i$ with some strict inequalities holding,
the bound for $\|\tilde A^{-1}\|_{\infty}$ can be smaller than that
for $\|A^{-1}\|_{\infty}$. On the other hand, it always holds that
$$
\|\tilde A\|_{\infty}\leq \|A\|_{\infty}.
$$
So, $\kappa_{\infty}(\tilde A)=\|\tilde A\|_{\infty}
\|\tilde A^{-1}\|_{\infty}$ is generally no more than $\kappa_{\infty}(A)=\|A\|_{\infty}
\|A^{-1}\|_{\infty}$ and may be considerably smaller than the latter, provided
that some dropped nonzero entries $a_{ij}$ from $A$ are not small.
For $A$ strictly column diagonally dominant, we define similar
$\tilde\beta_i$ and $\beta_i$ in the column sense. Then similar discussions
and the same claims can be made for the 1-norm condition number $\kappa_1(\tilde A)$
and $\kappa_1(A)$. The unique difference is
that the number of the corresponding $\tilde\beta_i>\beta_i$ in the column sense
is exactly $s$ since $A$ and $\tilde A$ only have $s$ different columns. Therefore,
for $A$ strictly (row or column) diagonally dominant, it is expected that
$\tilde A$ is better conditioned than $A$.

For $A$ an $M$-matrix, define
$$
r_i(A)=\sum_{j=1}^n a_{ij}, i=1,2,\ldots,n.
$$
It is known that $a_{ii}>0$ and $a_{ij}\leq 0$ for $j\not=i$ by the definition
of $M$-matrix. Then it is seen from $\beta_i$ defined above that
$r_i(A)=\beta_i,\ i=1,2,\ldots,n$. Similarly, we have $r_i(\tilde A)=\tilde\beta_i,\
i=1,2,\ldots,n$. It is proved in \cite{huliu} that if there is a positive diagonal
matrix $D$ such that
$$
r_i(AD)>0,\ i=1,2,\ldots,n
$$
then
$$
\frac{1}{\max r_i(AD)}\|D\|_{\infty}\leq\|A^{-1}\|_{\infty}\leq\frac{1}
{\min r_i(AD)}\|D\|_{\infty}
$$
and furthermore
$$
\|A^{-1}\|_{\infty}=\frac{1}
{r(AD)}\|D\|_{\infty}
$$
if $r_i(AD)=r(AD),\ i=1,2,\ldots,n$. Note that we have
$r_i({\tilde A}D)\geq r_i(AD)$ with some strict inequalities holding as $\tilde A$
drops some negative off-diagonal entries $a_{ij}$ from $A$ and $D$ is positive.
Thus, the upper bound for $\|{\tilde A}^{-1}\|_{\infty}$ is generally smaller than
that for $\|A^{-1}\|_{\infty}$.
Noticing that $\|\tilde A\|_{\infty}\leq \|A\|_{\infty}$,
we expect that $\kappa_{\infty}(\tilde A)$ is generally no more and can be smaller
than $\kappa_{\infty}(A)$. Similar discussions and claim go to $DA$ and
the 1-norm condition numbers of $A$ and $\tilde A$.

The fifth issue is on the existence of good sparse approximate inverses
of $\tilde A$.  The existence is generally definitive. We argue as follows: Since
${\mathcal P}(\tilde A)\subset {\mathcal P}(A)$,
for a given positive integer $l_{\max}$, we have
${\mathcal P}((|\tilde A|^T|\tilde A|)^{l_{\max}}\tilde A^T)
\subset {\mathcal P}((|A|^T |A|)^{l_{\max}}A^T)$
and ${\mathcal P}((I+|\tilde A|)^{l_{\max}})\subset {\mathcal P}((I+|A|)^{l_{\max}})$.
Assume that we take the initial sparsity ${\mathcal S}^{(0)}={\mathcal P}(I)$ and
implement SPAI and PSAI($tol$) $l_{\max}$ loops for $A$. Then it is
known \cite{huckle99,Jia2009power}
that the sparsity patterns of $M$ obtained by SPAI and PSAI($tol$)
are bounded by ${\mathcal P}((|A|^T |A|)^{l_{\max}}A^T)$
and ${\mathcal P}((I+|A|)^{l_{\max}})$, respectively. The same are true
for the sparsity patterns of $M$ obtained by SPAI and PSAI($tol$)
for $\tilde A$. This means that the effective envelops for the sparsity patterns
of $M$ obtained by SPAI and PSAI($tol$) for $\tilde A$ are contained
in those for $A$ for the same $l_{\max}$.
As a consequence, it is expected that $\tilde A$
has good sparse approximate inverses when $A$ does.

The last important issue is how to select stopping criteria for
Krylov iterations for $s+1$ linear systems so as to recover an approximate
solution of (\ref{eq:Axb}) with the prescribed accuracy.
It is seen from (\ref{update1}) that the solution $x$ of (\ref{eq:Axb}) is
formed from the ones of the $s+1$ new systems. Recall that the $s+1$ linear systems
are now supposed to be solved approximately by preconditioned Krylov solvers.
Our concerns are (i) how to define an approximate solution $\hat x$ from the $s+1$
approximate solutions of (\ref{regprob}) and (\ref{auxi}), and (ii) how accurately
we should solve (\ref{regprob}) and (\ref{auxi}) such that
$\hat x$ satisfies $\frac{\|r\|}{\|b\|}=\frac{\|b-A\hat x\|}{\|b\|} <
\varepsilon$. As it appears below, it is direct to settle down the first concern,
but the second concern is involved.

\begin{theorem}\label{thm2}
Let $\hat y$ and $\hat w_j,\ j=1,\ldots, s$, be the approximate solutions of
(\ref{regprob}) and (\ref{auxi}), respectively, and define
$\hat W=(\hat w_1,\ldots,\hat w_s)$ and the residuals
$r_{\hat y}=b-\tilde A\hat y$, $r_{\hat w_j}=u_j-\tilde A\hat w_j$. Assume
that $I+V^T\hat W$ is nonsingular with $V = (e_{j_1},e_{j_2},\ldots,e_{j_s})$,
and define $c=\|(I+V^T\hat W)^{-1}(V^T\hat y)\|$. Take
\begin{equation}\label{composi}
\hat x=\hat y-\hat W(I+V^T\hat W)^{-1}(V^T\hat y)
\end{equation}
to be an approximate solution of {\rm (\ref{eq:Axb})}. Then
if
\begin{equation}\label{stopb}
 \frac{\|r_{\hat y}\|}{\|b\|}<\frac{\varepsilon}{2}
\end{equation}
and
\begin{equation}\label{stopU2}
\frac{\|r_{\hat w_j}\|}{\|u_j\|}<\frac{\|b\|}{2\sqrt{s}c\|u_j\|}\varepsilon, \
j=1,2,\ldots,s,
\end{equation}
we have
\begin{equation}\label{stop}
 \frac{\|r\|}{\|b\|}=\frac{\|b-A\hat x\|}{\|b\|} < \varepsilon.
\end{equation}
\end{theorem}

\begin{proof}
Replacing $W$ and $y$ by their approximations $\hat W$ and
$\hat y$ in (\ref{update1}),
we get (\ref{composi}), which is naturally an approximate solution
of (\ref{eq:Axb}). Define $R_{\hat W}=U-\tilde A\hat W$. We obtain
\begin{align}
 r&=b-A\hat x=b-A\hat y+A\hat W(I+V^T\hat W)^{-1}(V^T\hat y) \nonumber\\
&=b-(\tilde A+UV^T)\hat y+(\tilde A+UV^T)\hat W(I+V^T\hat W)^{-1}
(V^T\hat y) \nonumber\\
&=r_{\hat y}-UV^T\hat y+UV^T\hat W(I+V^T\hat W)^{-1}(V^T\hat y)+
\tilde A\hat W(I+V^T\hat W)^{-1}(V^T\hat y) \nonumber\\
&=r_{\hat y}-U(I-V^T\hat W(I+V^T\hat W)^{-1})(V^T\hat y)+\tilde A
\hat W(I+V^T\hat W)^{-1}(V^T\hat y) \nonumber\\
&=r_{\hat y}-U(I-(I+V^T\hat W-I)(I+V^T\hat W)^{-1})(V^T\hat y)+\tilde A
\hat W(I+V^T\hat W)^{-1}(V^T\hat y) \nonumber\\
&=r_{\hat y}-U(I+V^T\hat W)^{-1}(V^T\hat y)+\tilde A\hat W(I+V^T\hat W)^{-1}
(V^T\hat y) \nonumber\\
&=r_{\hat y}-R_{\hat W}(I+V^T\hat W)^{-1}(V^T\hat y), \nonumber
\end{align}
from which it follows that
\begin{equation}\label{rry}
\|r\|\leq \|r_{\hat y}\|+c\|R_{\hat W}\|\leq \|r_{\hat y}\|+
c\|R_{\hat W}\|_F.
\end{equation}
By definition of $c$ and $R_{\hat W}(:,j)=r_{\hat w_j}$, we have
$\|R_{\hat W}\|_F=\sqrt{\sum_{j=1}^s\|r_{\hat w_j}\|^2}$. If (\ref{stopb}) and
$$
 \frac{\|r_{\hat w_j}\|}{\|b\|}<\frac{\varepsilon}{2\sqrt{s}c}, \ j=1,2,\ldots,s,
$$
it is seen from (\ref{rry}) that (\ref{stop}) holds.
Since the above relation is just (\ref{stopU2}), the theorem holds.
\end{proof}

We point out that because of (\ref{rry}) our stopping criterion (\ref{stopU2})
may be conservative. Furthermore,  $c$ is moderate if $I+V^T\hat W$ is well
conditioned, and it may be large if $I+V^T\hat W$ is ill conditioned. Since $s$
is supposed very small, this theorem indicates that the $s$ linear systems
(\ref{auxi}) need to be solved with the accuracy at the level of
$\varepsilon$. We may solve them  by Krylov solvers either simultaneously
in the parallel environment or independently in the sequential environment.
An alternative approach is to solve them using block Krylov solvers.
Note that $c$ cannot be computed until iterations for (\ref{regprob})
and (\ref{auxi}) terminate, but the stopping criterion (\ref{stopU2}) depends on $c$.
Therefore, the computation of $c$ and termination of iterations interacts.
In implementations, we simply replace $c$ by 1 in (\ref{stopU2}) and stop iterative
solvers for the $s$ systems (\ref{auxi}) with the modified accuracy requirement.
Because of this inaccuracy, (\ref{stop}) may fail to meet but $\|r\|/\|b\|$
should be at the level of $\varepsilon$.
In later numerical experiments, we will find that $c=1$ works very well and
makes (\ref{stop}) hold for almost all the test problems, and the
right-hand sides of (\ref{stop}) are only a little bit bigger than
$\varepsilon$ in the rare cases where (\ref{stop}) does not meet.

We make some further comments on (\ref{stopU2}). For a real-world
problem, if $A$ is poorly scaled, it is well known that
a preprocessing is generally done that uses scaling
to equilibrate $A$ so that its columns and/or rows are nearly the
same in norm. Without loss of generality, suppose
that $\|b\|$ is comparable to them in size; otherwise, we
replace $b$ by a scaled $\hat b=\alpha b$ with $\alpha$ a
scaling factor and instead solve the equivalent problem $A(\alpha x)=\hat b$,
so that $\|\hat b\|$ is comparable to the norms of columns of the
equilibrated $A$. Then the sizes of $\|b\|/\|u_j\|$ are typically
around $1$. We suppose that such processing is performed. As a result,
$s$ linear systems (\ref{auxi}) are solved with the accuracy at the level of
$\varepsilon$. Therefore, we need not worry about
the issue of small $\|b\|/\|u_j\|$ for a given problem.

\section{Numerical experiments}\label{numerexp}

In this section, we test our approach and compare it with the approach
that preconditions (\ref{eq:Axb}) by SPAI and PSAI($tol$) directly.
We report the numerical experiments
obtained by the Biconjugate Gradient Stablized (BiCGStab) with SPAI and PSAI($tol$)
preconditioning on (\ref{eq:Axb}) and (\ref{regprob}), (\ref{auxi}),
respectively. Such combinations give rise to four algorithms, and we name them
Standard-SPAI, New-SPAI, and Standard-PSAI($tol$) and New-PSAI($tol$),
abbreviated as S-SPAI, N-SPAI and S-PSAI($tol$), N-PSAI($tol$), respectively.

The experiments consists of three subsections, and
our aims are quadruple: (i) We demonstrate the
considerable efficiency superiority of N-SPAI to S-SPAI and that of N-PSAI($tol$)
to S-PSAI($tol$). (ii) With the same parameters used in SPAI,
we show that preconditioners obtained by N-SPAI
are more effective than the corresponding ones obtained by
S-SPAI. (iii) With the same parameters used in PSAI($tol$),
we illustrate that the preconditioners by S-PSAI($tol$) and N-PSAI($tol$)
are equally effective for preconditioning each problem, provided that
they can be computed. (iv) We illustrate that if the numbers of nonzero entries
of preconditioners, i.e., the sparsity of preconditioners, are (almost) the same
then PSAI($tol$) is more effective than SPAI for preconditioning
both irregular and regular sparse linear systems. The results mean
that PSAI($tol$) captures a sparsity pattern of $A^{-1}$
and $\tilde A^{-1}$ more effectively than SPAI and thus generate better
preconditioners. They also imply that even for regular sparse linear systems,
SPAI may be ineffective for preconditioning.

We mention that, for other Krylov solvers, such as BiCG, CGS and the restarted
GMRES(20), we have done similar numerical experiments and had the same findings
as above. So it suffices to only report and evaluate the results obtained by BiCGStab.

Before testing our approach, we look into all the matrices in the University of
Florida sparse matrix collection \cite{davis2011university} and give illustrative
information on where irregular sparse linear systems
come from, how common they are in practice, how big $s$ can be and how dense
irregular columns. We divide matrices into their problem
domains, and sort them by percentage of matrices in that domain that are irregular.
Table \ref{tabl:irr} lists the relevant information, where ``per. irreg''
denotes the percentage of irregular matrices in each domain, ``\#reg. prob'' and
``\#irreg. prob'' are the numbers of regular and irregular matrices in each domain,
respectively. Matrices labeled as ``graphs" in the collection are excluded, many
of which are irregular, but not all are linear systems.

\begin{table}[ht]{\footnotesize
\begin{center}
\caption{Statistics of regular/irregular problems in the sparse matrix collection
\cite{davis2011university}}\label{tabl:irr}
\begin{tabular}{|c|c|c|c|}
\hline
problem domain&per. irreg&\#reg. prob&\#irreg. prob\\
\hline
frequency-domain circuit simulation problem&100\%&0&4\\
linear programming problem&100\%&0&1\\
semiconductor device problem&63\%&13&22\\
optimization problem&61\%&53&82\\
power network problem&56\%&26&33\\
circuit simulation problem&51\%&124&132\\
computer graphics/vision problem&33\%&2&1\\
economic problem&33\%&44&21\\
counter-example problem&25\%&6&2\\
eigenvalue/model reduction problem&20\%&28&7\\
material problem&11\%&25&3\\
chemical process simulation problem&11\%&62&8\\
statistical/mathematical problem&11\%&8&1\\
theoretical/quantum chemistry problem&11\%&54&7\\
2D/3D problem&10\%&118&13\\
acoustics problem&8\%&12&1\\
structural problem&5\%&287&14\\
combinatorial problem&3\%&28&1\\
computational fluid dynamics problem&3\%&167 &5\\
thermal problem&3\%&30&1\\
electromagnetic  problem&2\%&49&1\\
least squares problem&0\%&2&0\\
model reduction problem&0\%&47&0\\
other problem&0\%&4&0\\
random 2D/3D problem&0\%&2&0\\
robotics problem&0\%&3&0\\
\hline
\end{tabular}
\end{center}}
\end{table}

The statistics in Table~\ref{tabl:irr} illustrates that irregular sparse linear
systems are quite common and come from many applications.
Figure~\ref{fignew} depicts many more details. In the top left plot, a circle is a
matrix in the collection, the $x$ axis is the order of a square matrix, and
the $y$ axis is the number $s$ of dense columns. The steep line is $s=n$, which
is not achievable, and the flat line is $s=\sqrt{n}$.
This figure plots matrices with at least one dense column.
In the top right plot, each dot is a square matrix, the $x$ axis is the mean,
i.e., the average number $p$, of nonzero entries in each column, which equals
$\lfloor nnz (A) /n\rfloor$, and the $y$ axis is the number of
nonzero entries in the densest column divided by the mean for that matrix.
The line parallel to the $x$ axis is $y=10$.
Matrices have at least one dense column if
they reside above the line $y=10$. For each $x$, the bigger $y$, the denser
the irregular column.  The bottom two plots are the same,
but with social networks and other graphs excluded. \footnote{The figures and the
previous data analysis are due to Professor Davis, and we thank him very much.}

These figures further demonstrate that there are very dense columns for many matrices,
irregular sparse matrices are common, and $s$ can be quite big.

\begin{figure}[ht]
\begin{center}
\epsfig{figure=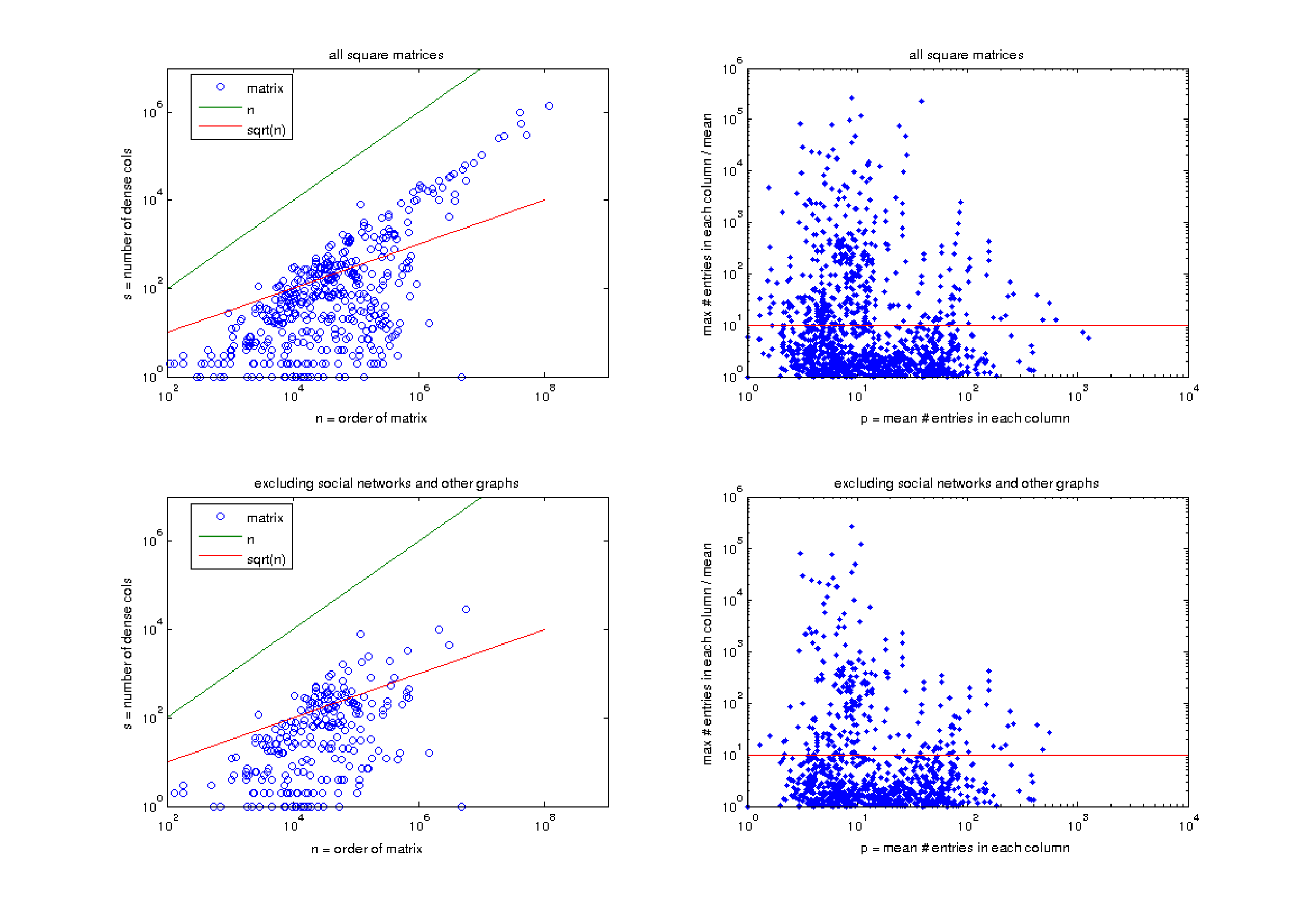,height=3.5in}
\end{center}
\caption{Illustrations of irregular matrices in the sparse matrix collection
\cite{davis2011university}.}
\label{fignew}
\end{figure}

We test our approach on some of the above irregular sparse linear systems.
A brief description is presented in Table \ref{prob}, where
$\kappa(A)=\|A\|\|A^{-1}\|$. The right-hand side $b$ of $Ax=b$ was formed
by taking the solution $x=(1,1,\ldots,1)^T$. Taking the initial approximate
solution to be zero for each problem and $\varepsilon=10^{-8}$ in stopping
criterion (\ref{stop}), we have found that BiCGStab without preconditioning did
not converge for any of the test problems within 500 iterations. We note that
{\em cbuckle} is symmetric positive definite.
So it should be better to design preconditioners to maintain the symmetry
of preconditioned matrices, which
can be achieved by using factorized or splitting preconditioners, e.g.,
\cite{benzi1996sparse, janna2010block}, so that some more efficient symmetric
Krylov solvers, e.g., the Conjugate Gradient (CG) or Minimal Residual (MINRES)
method, can be applied. In our experiments, however, we used {\em cbuckle} in
PSAI($tol$) purely for test purposes and treated it as a general matrix.

\begin{table}[ht]{\footnotesize
\begin{center}{
\caption{The description of test matrices. ``sym'' denotes the
symmetry of a matrix, and we used the {\sc Matlab} function {\sf condest} to
estimate the 1-norm condition numbers of the latter seven larger matrices.}
\label{prob}
\begin{tabular}{@{}lccccl@{}}
 \hline
Matrix&$n$&$nnz(A)$&$\kappa(A)$&sym&Description\\
\hline
fs\_541\_3&541&4282&$2.83\times10^{11}$&No&2D/3D problem\\
fs\_541\_4&541&4273&$1.17\times 10^{10}$&No&2D/3D problem\\
rajat04&1041&8725&$1.64\times 10^{8}$&No&circuit simulation problem\\
rajat12&1879&12818&$6.91\times 10^{5}$&No&circuit simulation problem\\
tols4000&4000&8784&$2.36\times 10^7$ &No&computational fluid dynamics problem\\
cbuckle&13681&676515&$3.30\times 10^7$&Yes&structural problem\\
ASIC\_100k&99340&940621&$1.46\times 10^{11}$&No&circuit simulation problem\\
dc1&116835&766396&$1.01\times 10^{10}$&No&circuit simulation problem\\
dc2&116835&766396&$8.86\times 10^9$&No&circuit simulation problem\\
dc3&116835&766396&$1.16\times 10^{10}$&No&circuit simulation problem\\
trans4&116835&749800&$3.30\times 10^9$&No&circuit simulation problem\\
trans5&116835&749800&$2.32\times 10^9$&No&circuit simulation problem\\
\hline
\end{tabular}}
\end{center}}
\end{table}

We conducted the numerical experiments on an Intel
(R) Core (TM)2 Quad CPU E8400 @ $3.00$GHz with main memory 2 GB
under the Linux operating system. The computations were done using
{\sc Matlab} 7.8.0 with the machine precision $\epsilon_{\rm mach}=2.22\times
10^{-16}$, and SPAI preconditioners were constructed by the
SPAI 3.2 package \cite{barnard} of Barnard, Br\"{o}ker, Grote and
Hagemann, which is written in C/MPI. Our PSAI($tol$) code is
written in {\sc Matlab} language for the sequential environment.
We used both SPAI and PSAI($tol$) as right preconditioning and
took the initial patterns $\mathcal S_k^{(0)}=\{k\},
\ k=1,2,\ldots,n$. We took $c=1$ in (\ref{stopU2}) and stopped Krylov iterations
when (\ref{stopb}) and (\ref{stopU2}) were satisfied with $\varepsilon=10^{-8}$
or 500 iterations were used. The initial approximate solution for each
problem was zero vector. With $\hat x$ defined by (\ref{composi}),
we computed the actual relative residual norm
\begin{equation}\label{rrnorm}
rr=\frac{\|b-A\hat x\|}{\|b\|}
\end{equation}
and compared it with the required accuracy $\varepsilon=10^{-8}$.

Before performing our algorithms, we carried out row Dulmage--Mendelsohn
permutations \cite{Duff,Pothen} on the matrices that have zero diagonals,
so as to make their diagonals nonzero. This preprocessing, though not necessary
theoretically, may be more suitable for taking
the initial pattern ${\mathcal S}={\mathcal P}(I)$ in SPAI, which always
retains the diagonals as most profitable indices in $M$.
The related {\sc Matlab} commands are $j=\verb"dmperm"(A)$ and $A=A(j,:)$.
We applied \verb"dmperm" to {\em rajat04}, {\em rajat12}, {\em tols4000} and
{\em ASIC\_100k}. For all the matrices listed in Table \ref{prob}, $\tilde A$
is constructed by retaining the diagonal entry and $p-1$ nonzero entries
nearest to the diagonal and dropping the others in each of the irregular columns
of $A$, and $U$ is composed of the dropped nonzero entries.
Table~\ref{matnnz} shows some useful information on each $A$ and
$\tilde A$.

\begin{table}[ht]{\footnotesize
\begin{center}{\small
\caption{Some information on $A$ and $\tilde A$. $s$: \# irregular columns
in $A$; $p$: the average number of nonzero entries per column of $A$; $p_d$:
\# nonzero entries in the densest column of $A$. We used {\sc Matlab} function
{\sf condest} to estimate the 1-norm condition numbers of the last seven
larger matrices.}
\label{matnnz}
\begin{tabular}{|c|c|c|c|c|c|c|c|c|}
\hline
&$n$&$s$&$p$&$p_d$&$nnz(A)$&$nnz(\tilde A)$&$\kappa(A)$&$\kappa(\tilde A)$\\
\hline
fs\_541\_3&541&1&7&538&4282&3745&$2.83\times10^{11}$&$7.86\times10^6$\\ \hline
fs\_541\_4&541&1&7&535&4273&3739&$1.17\times 10^{10}$&$1.64\times10^6$\\ \hline
rajat04&1041&4&8&642&8725&7306&$1.64\times 10^{8}$&$1.16\times10^8$\\ \hline
rajat12&1879&7&6&1195&12818&9876&$6.91\times 10^{5}$&$6.53\times10^5$\\ \hline
tols4000&4000&18&2&22&8784&8424&$2.36\times 10^7$&$2.36\times 10^7$\\ \hline
cbuckle&13681&1&49&600&676515&675916&$3.30\times 10^7$&$8.06\times10^7$\\ \hline
ASIC$\ast$&99340&122&9&92258&940621&742736&$1.46\times10^{11}$&$9.28\times10^9$\\ \hline
dc1&116835&55&6&114174&766396&595757&$1.01\times 10^{10}$&$2.17\times10^8$\\ \hline
dc2&116835&55&6&114174&766396&595757&$8.86\times 10^9$&$5.81\times10^7$\\ \hline
dc3&116835&55&6&114174&766396&595757&$1.16\times 10^{10}$&$1.52\times10^8$\\ \hline
trans4&116835&55&6&114174&749800&587459&$3.30\times 10^{9}$&$3.46\times10^8$\\ \hline
trans5&116835&55&6&114174&749800&587459&$2.32\times 10^{9}$&$6.54\times10^7$\\
\hline
\end{tabular}}
\end{center}}
\end{table}

We have some informative observations from Table~\ref{matnnz}. As is seen, except
{\em cbuckle} and {\em tols4000}, all the other test matrices have some almost
fully dense columns. The matrix {\em cbuckle} and {\em tols4000} have 1 and 18
not very dense irregular columns, respectively.
Precisely, except {\em cbuckle} and {\em tols4000}, {\em rajat04} and {\em rajat12}
have some irregular columns which have more than $n/2$ nonzero entries, and all
the other matrices have some fully dense columns. The table also shows the number
$s$ of irregular columns of each $A$ and the percentage
$s/n$. The biggest two percentages $0.45\%$ and $0.38\%$ correspond to
{\em tols4000} and {\em rajat04}, where $s=18$ and $4$, respectively.
It is remarkable that the eight $\tilde A$ are considerably
better conditioned than the corresponding $A$, and their condition
numbers are accordingly reduced by roughly one to five orders.
For the other three matrices {\em rajat04}, {\em rajat12} and
{\em cbuckle}, each pair of $\tilde A$ and $A$ have very near condition numbers.

We next investigate the patterns of entries of large magnitude
in the ``exact" inverses of an
irregular sparse matrix and the regular sparse matrix induced from it.
We only take {\em rajat04} as an example. After performing row
Dulmage--Mendelsohn permutation on it, we use the {\sc Matlab} function
$\verb"inv"$ to compute $A^{-1}$ and $\tilde A^{-1}$ and then drop their
nonzero entries whose magnitudes fall below $10^{-3}$. We depict the
patterns of sparsified $A^{-1}$ and $\tilde A^{-1}$ as (A) and (B) in
Figure~\ref{fig1}, respectively. It is clear that good approximate inverses
of $A$ and $\tilde A$ are sparse but there are
several dense columns in (A), which means that an effective sparse
approximate inverse of $A$ is irregular. On the other hand, the situation is
improved substantially in (B), from which it is seen that a good sparse
approximate inverse of $\tilde A$ is sparser than
the matrix in (A). These results typically demonstrate that
good sparse approximate inverses of an irregular sparse matrix are
also irregular, while the resulting regular sparse matrix $\tilde A$
has sparser approximate inverses. We have checked several other test matrices and
have had the same findings.
%Perceptually, all of these have strongly supported
%our empirical assertions that the good sparse approximate inverses of an
%irregular sparse matrix are generally irregular too.

\begin{figure}[ht]{\small
\begin{center}
\epsfig{figure=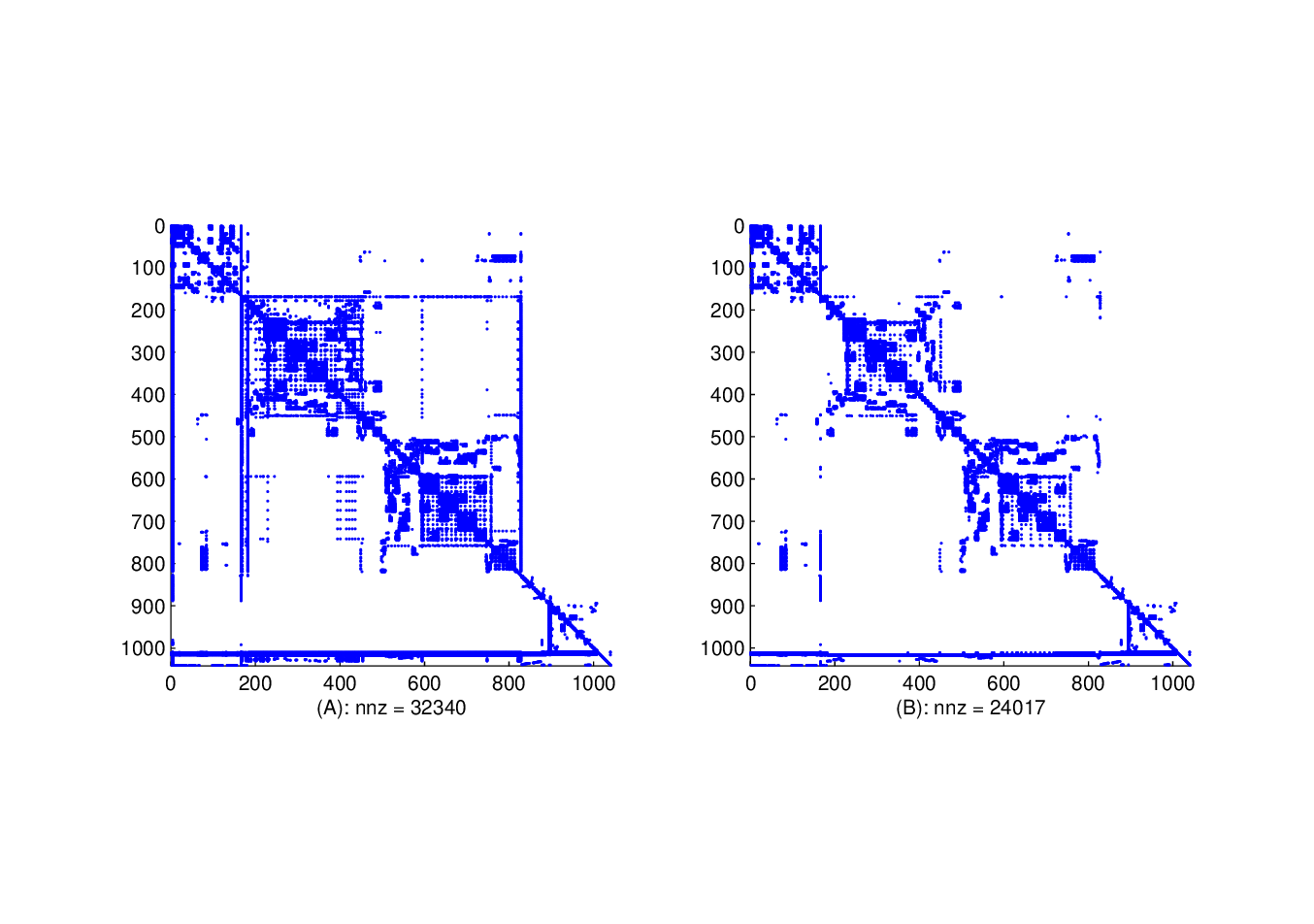,height=3.5in}
\end{center}
\caption{{\em rajat04}: (A) and (B) are the patterns of sparsified $A^{-1}$ and
$\tilde A^{-1}$, respectively.}\label{fig1}}
\end{figure}

In all the later tables, we denote by $T_{setup}$ and $T_{solve}$ the setup
(construction) time(sec) of $M$ and the time(sec) of solving the
preconditioned linear systems by BiCGStab, respectively, by $spar=nnz(M)/nnz(A)$ or
$nnz(M)/nnz(\tilde A)$  the sparsity of $M$ relative
to $A$ or $\tilde A$, and by $iter$ the iteration number
that BiCGStab used for (\ref{eq:Axb}) and the maximum of
iteration numbers that BiCGStab used for the $s+1$ systems
(\ref{regprob}) and (\ref{auxi}), respectively. The actual relative residual
norm $rr$ defined by (\ref{rrnorm}) is $a\cdot\varepsilon$, and we only
list the multiple $a$ in the tables. So $a<1$ indicates that BiCGStab converged
with the prescribed accuracy $\varepsilon$. The size of $a$ reflects
whether taking $c=1$ in (\ref{stopU2}) is reliable or not.

\subsection{Numerical results obtained by SPAI}\label{expSPAI}

We take $\delta=0.4$ as the stopping criterion for SPAI. Since effective
sparsity patterns of $A^{-1}$ and $\tilde{A}^{-1}$ are generally unknown
in advance, some key parameters, especially
the number of most profitable indices per loop involved can only be chosen
empirically in order to control the sparsity of $M$ and simultaneously to
make $M$ achieve the desired accuracy $\delta$ as much as possible. In the
experiments of this subsection, we fix the number of most profitable indices
to be 5 per loop, which is also the default value in the SPAI 3.2 code.
The maximum $l_{\max}$ of loops is set to 20, bigger than the default value 5
in the code. SPAI terminated whenever $\|Am_k-e_k\|\leq\delta,
\ k=1,2,\ldots,n$ or $l_{\max}$ loops were attained.
SPAI is run for $\tilde{A}$ with the same parameters. With the given
parameters and the initial pattern ${\cal S}_k^{(0)}=\{k\},\ k=1,2,\ldots,n$,
the number of nonzero entries per column in the final $M$ is bounded by
$1+5\times 19=96$. So if good preconditioners are irregular sparse and there
is at least one column whose number of nonzero entries are bigger than 96,
then SPAI may be ineffective for preconditioning. We attempt to show that
with the same parameters used in SPAI, N-SPAI is much more efficient than S-SPAI,
and the former is more effective than the latter for preconditioning.
We mention that we have omitted the symmetric positive definite matrix
{\em cbuckle} since the matrix input in the SPAI 3.2 package is only supported
by ``real general'' Matrix Market coordinate format.
Table~\ref{spaidata} lists the results.

%%==== This table is to report results of SPAI ('-ep 0.4 -mn 5 -ns 20 ')=======
\begin{table}{\footnotesize
\begin{center}
\caption{Numerical results obtained by SPAI.
$T_{setup}$ and $T_{solve}$: the setup time of $M$
and the time of BiCGStab iterations for the preconditioned linear
systems, respectively. The relative residual
norm $rr=a\cdot\varepsilon$. $n_c$: \# columns in $M$
failing to satisfy the accuracy $\delta=0.4$. For S-SPAI, $spar=nnz(M)/nnz(A)$
and $iter$: \# iterations of BiCGStab for (\ref{eq:Axb}), while for
N-SPAI, $spar=nnz(M)/nnz(\tilde A)$ and $iter$: the maximum of \# iterations of
BiCGStab for the $s+1$ systems (\ref{regprob}) and (\ref{auxi}).}
\label{spaidata}
\begin{tabular}{|p{.9cm}<{\centering}|p{.7cm}<{\centering}|p{.7cm}<
{\centering}|p{.5cm}<{\centering}|p{.4cm}<{\centering}|p{.4cm}<{\centering}
|p{.4cm}<{\centering}|p{.7cm}<{\centering}|p{.7cm}<{\centering}|p{.5cm}<
{\centering}|p{.4cm}<{\centering}|p{.4cm}<{\centering}|p{.4cm}<{\centering}|}
\hline
&\multicolumn{6}{c|}{S-SPAI}&\multicolumn{6}{c|}{N-SPAI}\\ \cline{2-13}
&$T_{setup}$&$T_{solve}$&$spar$&$iter$&$a$&$n_c$&$T_{setup}$&$T_{solve}$&$spar$&
$iter$&$a$&$n_c$\\ \hline
fs\_541\_3&0.33&0.01&1.34&21&0.40&1 &0.17&0.01&1.52&8&0.62&0\\ \hline
fs\_541\_4&0.14&0.01&0.81&14&0.01&1 &0.04&0.01&0.91&6&0.68&0\\ \hline
rajat04&1.17&0.01&0.36&31&0.20&6 &0.16&0.02&0.40&17&0.57&2\\\hline
rajat12&3.19&0.01&0.86&48&0.31&3 &0.39&0.08&1.08&39&0.80&0\\\hline
tols4000&0.10&0.02&1.01&3&0.44&6 &0.05&0.02&1.02&4&0.48&6\\\hline
ASIC$\ast$&$*$&$*$&$*$&$*$&$*$&$*$&1619&13.52&0.66&10&1.37&12\\ \hline
dc1&$*$&$*$&$*$&$*$&$*$&$*$&8216&44.07&1.71&351&0.61&1\\\hline
dc2&$*$&$*$&$*$&$*$&$*$&$*$&6210&41.75&1.62&76&0.75&0\\ \hline
dc3&$*$&$*$&$*$&$*$&$*$&$*$&5243&47.05&1.63&98&0.91&9\\ \hline
trans4&$*$&$*$&$*$&$*$&$*$&$*$&3798&9.13&1.64&39&0.49&0\\ \hline
trans5&$*$&$*$&$*$&$*$&$*$&$*$&3546&20.79&1.57&84&0.65&0\\\hline
\end{tabular}
\end{center}}
\end{table}

We make some comments on the results. First, we look at the efficiency of
constructing $M$ by S-SPAI and N-SPAI. In the table, the notation
``$*$'' for the last six larger matrices indicates that S-SPAI could not
compute $M$ within 100 hours because of the irregular sparsity of $A$.
Since each of these six matrices has fully dense irregular columns,
the unaffordable time consumption results from the large
cardinal numbers of $\tilde{\mathcal J}$ and $\mathcal L$ in
\S\ref{spai}. More precisely, for a fully dense irregular column
of $A$, we have to compute almost the $n$ numbers $\mu_j$
in (\ref{muj}), sort almost $n$ indices in $\tilde{\mathcal L}$ and select
five most profitable indices among them at each loop. Carrying out these tasks
is very time consuming. In contrast, N-SPAI did a very good job due to
the regular sparsity of $\tilde A$. For the other matrices and given the
parameters, the two $M$ obtained by S-SPAI and N-SPAI have very
similar sparsity, but it is seen from $T_{setup}$ that N-SPAI can be
considerably more efficient than S-SPAI, and the former can be several
times faster, e.g., six and seven times faster for {\em rajat04} and
{\em rajat12}, respectively. For the last six larger matrices,
N-SPAI computed all the $M$ within no more than half an hour to two hours,
drastic improvements over S-SPAI!

Second, we investigate the effectiveness of $M$ for preconditioning.
To be more illuminating, we recorded the number $n_c$ of columns in $M$
which failed to satisfy the accuracy $\delta=0.4$ for each matrix in
S-SPAI or N-SPAI. We have observed that for the first four matrices,
the $n_c$ produced by N-SPAI are always smaller than those by S-SPAI.
This illustrates that it is more difficult for SPAI to compute an
effective preconditioner when $A$ is irregular sparse. This is also
reflected by the iteration number $iter$, from which it is seen that N-SPAI is
more effective than S-SPAI for the first four matrices. This confirms
our claim that SPAI may be less effective for preconditioning irregular
sparse linear systems.

Third, it is seen from the quantity $a$ that most of the actual
relative residual norms $rr$ in (\ref{rrnorm}) dropped below $\varepsilon$.
The case that $a>1$ happened only for {\em ASIC\_100k},
but the actual relative residual norm $rr=1.37\times 10^{-8}$
is very near to the required accuracy $\varepsilon=10^{-8}$,
indicating that the algorithm essentially converged. This means that taking
$c=1$ in (\ref{stopU2}) worked robustly in practice.

Finally, we compare the overall performance of N-SPAI and S-SPAI. For the last six
larger matrices, S-SPAI failed to compute the desired $M$ within 100 hours, while
N-SPAI was very efficient to do the same job and exhibited huge superiority.
For the other problems, N-SPAI is two to eight times faster than S-SPAI.
We see that the construction of $M$ by N-SPAI dominates the
total cost of our approach, and the efficiency of N-SPAI compensates for the
price of solving the $s+1$ linear systems. As a result, as far as the overall
efficiency is concerned, our approach is considerably superior to SPAI applied to
precondition the irregular sparse (\ref{eq:Axb}) directly.

\subsection{Numerical results obtained by PSAI($tol$)}

We look into the performance of N-PSAI($tol$) and S-PSAI($tol$) and show that the
former is much more efficient than the latter. We also illustrate
that N-PSAI($tol$) and S-PSAI($tol$) are equally effective for
preconditioning, provided that they can compute preconditioners with
the same parameters used in PSAI($tol$). We always take $\delta=0.4$
and $l_{\max}=10$ to control the sparsity and quality of $M$ for both S-PSAI($tol$)
and N-PSAI($tol$). PSAI($tol$) terminated when $\|Am_k-e_k\|\leq\delta$ or the loops
$l>l_{\max}$. We used the $tol_k$ defined by (\ref{tolk}) as the dropping tolerances.
In Table~\ref{psaidata}, we report the results on the matrices in Table~\ref{prob},
where $l_m$ is the actual maximum loops that PSAI($tol$) used.

%%==== This table was to report results of PSAI(tol) (ep = 0.4, l_max=10), where
%% l_m denotes the maximum of the outer loop steps in PSAI(tol), for each matrix.
\begin{table}[ht]{\footnotesize
\begin{center}
\caption{Numerical results obtained by PSAI($tol$).
$T_{setup}$ and $T_{solve}$: the setup time of $M$
and the time of BiCGStab iterations for the preconditioned linear
systems, respectively. The relative residual
norm $rr=a\cdot\varepsilon$. $l_m$: the actual maximum loops. For S-PSAI($tol$),
$spar=nnz(M)/nnz(A)$ and $iter$: \# iterations of BiCGStab for (\ref{eq:Axb}),
while for N-PSAI($tol$), $spar=nnz(M)/nnz(\tilde A)$ and $iter$: the maximum of
\# iterations of BiCGStab for the $s+1$ systems (\ref{regprob}) and (\ref{auxi}).}
\label{psaidata}
{\small
\begin{tabular}{|p{1cm}<{\centering}|p{.7cm}<{\centering}|p{.7cm}<{\centering}
|p{.5cm}<{\centering}|p{.4cm}<{\centering}|p{.4cm}<{\centering}|
p{.4cm}<{\centering}|p{.7cm}<{\centering}|p{.7cm}<{\centering}|p{.5cm}<
{\centering}|p{.4cm}<{\centering}|p{.4cm}<{\centering}|p{.4cm}<{\centering}|}
\hline
&\multicolumn{6}{c|}{S-PSAI($tol$)}&\multicolumn{6}{c|}{N-PSAI($tol$)}\\ \cline{2-13}
&$T_{setup}$&$T_{solve}$&$spar$&$iter$&$a$&$l_m$&$T_{setup}$&$T_{solve}$&$spar$
&$iter$&$a$&$l_m$\\ \hline
fs\_541\_3&0.36&0.01&1.55&6&0.31&3  &0.29&0.01&1.63&6&0.30&4\\ \hline
fs\_541\_4&0.30&0.01&1.35&5&0.22&3  &0.25&0.01&1.38&5&0.42&3\\ \hline
rajat04&1.27&0.01&0.72&11&0.79&4  &0.28&0.02&0.51&11&0.58&4\\ \hline
rajat12&3.48&0.01&2.22&32&0.74&2  &1.40&0.10&1.92&44&0.56&2\\ \hline
cbuckle&3841&1.62&3.41&85&0.21&5  &2322&2.12&2.76&71&0.55&5\\ \hline
tols4000&0.94&0.01&0.97&2&0.11&1 &0.79&0.01&1.00&2&0.23&2\\ \hline
ASIC$\ast$&-&-&-&-&-&-&1429&10.05&0.81&6&0.78&2\\ \hline
dc1&-&-&-&-&-&-&3203&31.36&1.54&201&0.81&3\\ \hline
dc2&-&-&-&-&-&-&2975&27.33&1.62&59&0.58&3\\ \hline
dc3&-&-&-&-&-&-&2966&24.10&1.63&56&0.61&8\\ \hline
trans4&-&-&-&-&-&-&3073&5.61&1.79&31&0.44&4\\ \hline
trans5&-&-&-&-&-&-&2856&9.86&1.61&58&1.29&4\\ \hline
\end{tabular}
}
\end{center}}
\end{table}

From Table~\ref{psaidata} we first see that all $l_m<l_{\max}=10$. This indicates
that S-PSAI($tol$) and N-PSAI($tol$) computed the $M$ with the prescribed accuracy
$\delta$.

In the table, the notation ``-'' for the last six larger matrices
indicates that our computer was out of memory when constructing
each $M$. The cause is that each $A$ of them has some fully dense
irregular columns, which result in some large LS problems (\ref{lsprob}).
So, PSAI($tol$) encounters a severe difficulty when $A$ has some very
dense columns. For all the other matrices, S-PSAI($tol$)
generated $M$ with the desired accuracy $\delta=0.4$.
Furthermore, we have seen that S-PSAI($tol$) always used nearly the same $l_m$
as N-PSAI($tol$) for each $A$, but it is considerably more time consuming than
N-PSAI($tol$). So, using (almost) the same loops $l_m$, PSAI($tol$) captured
the effective sparsity patterns of both irregular sparse $A$
and regular sparse $\tilde A$. So they are expected to be comparably
effective for preconditioning both irregular and regular sparse linear problems.
Indeed, as we have seen from Table~\ref{psaidata}, both S-PSAI($tol$) and
N-PSAI($tol$) provide effective preconditioners for the first five matrices
since BiCGStab used comparable iterations to achieve the convergence.
Compared with the results in \S\ref{spai}, we see this is a distinctive feature that
S-SPAI and N-SPAI do not have, where SPAI may be ineffective when $A$ is irregular
sparse, and the situation may be improved when $A$ is regular sparse.

Still, we see that $c=1$ works very well for all problems and makes almost all
the actual relative residual norms defined by (\ref{rrnorm}) drop $\varepsilon$.
The only exception is for {\em trans5}, where $a=1.29$. But such $a$ indicates
that the actual relative residual norm for this problem is very near to
$\varepsilon$, so we may well accept the approximate solution as
essentially converged.

Finally, we are concerned with the overall performance of the algorithms.
The table clearly shows that the performance of N-PSAI($tol$) is superior
to S-PSAI($tol$), in terms of the total time equal to $T_{setup}$ plus
$T_{solve}$. Since S-PSAI($tol$) and N-PSAI($tol$) provide equally
effective preconditioners with comparable sparsity for each $A$ and $\tilde A$,
it is natural that the time $T_{solve}$ of Krylov iterations for N-PSAI($tol$)
is more than that of S-PSAI($tol$) as N-PSAI($tol$)
solves the $s+1$ linear systems. Even so, however, the time $T_{solve}$ of
BiCGStab iterations for the $s+1$ systems is negligible, compared to the setup
time $T_{setup}$ of the preconditioner $M$ for each problem.

\subsection{Effectiveness comparison of S-SPAI and S-PSAI($tol$) and that of
N-SPAI and N-PSAI($tol$)}

We attempt to give some comprehensive comparison of the preconditioning
effectiveness of S-SPAI and S-PSAI($tol$) and that of N-SPAI and N-PSAI($tol$).
To do so, we take the sparsity of $M$ as a reasonable standard.
As mentioned previously, all the $M$ obtained by PSAI($tol$)
in Table~\ref{psaidata} satisfy the desired accuracy $\delta=0.4$. Now,
for each $A$, we adjust the number of most profitable
indices per loop and the maximum $l_{\max}$ of loops in the SPAI code, so that
the sparsity of $M$ generated by N-SPAI is approximately equal to that of the
corresponding preconditioner obtained by N-PSAI($tol$) in Table~\ref{psaidata}.
We then perform S-SPAI with these parameters to compute a sparse approximate
inverse of $A$. We aim to show that, {\em given a similar sparsity of
preconditioners}, N-PSAI($tol$) may capture sparsity patterns of $\tilde A^{-1}$
more effectively and thus produces better preconditioners than N-SPAI
for regular sparse matrices.
We also show that there are the same findings for S-PSAI($tol$) and S-SPAI.
These demonstrate that PSAI($tol$) captures the sparsity patterns
of $A^{-1}$ and $\tilde A^{-1}$ more effectively and are thus more
effective for preconditioning than SPAI. The cause should be due to the
fact that PSAI($tol$) does the job in a globally optimal sense while SPAI does it
in a locally optimal sense, thereby confirming  our comments in the beginning of
\S\ref{overview}.

To make each $M$ by N-SPAI as (almost) equally sparse as that obtained by
N-PSAI($tol$) for each $\tilde A$, we take the parameters in the SPAI 3.2
code as\\
\texttt{
\begin{tabular}{ll}
fs\_541\_3&\hspace{1cm}'-mn 6 -ns 15'\\
fs\_541\_4&\hspace{1cm}'-mn 6 -ns 20'\\
rajat04 &\hspace{1cm}'-mn 8 -ns 20'\\
rajat12 &\hspace{1cm}'-mn 12 -ns 10'\\
tols4000&\hspace{1cm}'-mn 5 -ns 20'\\
ASIC\_100k&\hspace{1cm}'-mn 6 -ns 2'\\
dc1 &\hspace{1cm}'-mn 5 -ns 7'\\
dc2 &\hspace{1cm}'-mn 5 -ns 7'\\
dc3 &\hspace{1cm}'-mn 5 -ns 7'\\
trans4&\hspace{1cm}'-mn 6 -ns 15'\\
trans5&\hspace{1cm}'-mn 6 -ns 7'
\end{tabular}
}\\
where $ns=l_{\max}$ in our notation with '-ns\ j' denoting $ns=j$,
and $mn$ is the number of most profitable
indices per loop with '-mn\ j' denoting $mn=j$. Table~\ref{spaidata2} reports
the results.

\begin{table}{\footnotesize
\begin{center}
\caption{Numerical results obtained by SPAI. $T_{setup}$ and $T_{solve}$: the setup
time of $M$ and the time of BiCGStab iterations for the preconditioned linear
systems, respectively. The relative residual norm $rr=a\cdot\varepsilon$. For
S-SPAI, $spar=nnz(M)/nnz(A)$, $iter$: \# iterations of BiCGStab for (\ref{eq:Axb}),
while for N-SPAI, $spar=nnz(M)/nnz(\tilde A)$ and $iter$: the maximum of
\# iterations of BiCGStab for the $s+1$ systems (\ref{regprob}) and (\ref{auxi}).}
\label{spaidata2}
\begin{tabular}{|c|c|c|c|c|c|c|c|c|c|c|}
\hline
&\multicolumn{5}{c|}{S-SPAI}&\multicolumn{5}{c|}{N-SPAI}\\ \cline{2-11}
&$T_{setup}$&$T_{solve}$&$spar$&$iter$&$a$&$T_{setup}$&$T_{solve}$&$spar$&$iter$&$a$\\
\hline
fs\_541\_3&0.32&0.01&1.65&48&0.72 &0.20&0.01&1.64&20&0.40\\ \hline
fs\_541\_4&0.15&0.01&0.96&15&0.11  &0.03&0.01&0.94&7&0.34\\ \hline
rajat04&1.32&0.01&0.56&20&0.99  &0.16&0.02&0.50&12&0.77\\\hline
rajat12&2.71&0.02&2.05&42&0.39   &0.42&0.09&1.99&29&0.67\\\hline
tols4000&0.10&0.02&1.01&3&0.44 &0.05&0.02&1.02&4&0.48\\\hline
ASIC\_100k&$78$h&$10.16$&$0.65$&$500$&$1560$
& 1807&12.61&0.76&10&0.92\\ \hline %'-mn 6 -ns 2 -mb 200000'
dc1&$*$&$*$&$*$&$*$&$*$&7275&88.25&1.66&500&0.73\\\hline
dc2&$*$&$*$&$*$&$*$&$*$&6208&54.33&1.62&90&0.57\\ \hline
dc3&$*$&$*$&$*$&$*$&$*$&5198&70.63&1.63&422&1.36\\ \hline
trans4&$*$&$*$&$*$&$*$&$*$&3571&9.23&1.82&40&0.47\\ \hline
trans5&$*$&$*$&$*$&$*$&$*$&3033&19.52&1.71&81&0.82\\\hline
\end{tabular}
\end{center}}
\end{table}
Based on Tables~\ref{psaidata}--\ref{spaidata2}, we next compare the
preconditioning effectiveness of S-SPAI and S-PSAI($tol$) and that
of N-SPAI and N-PSAI($tol$), respectively.

As $iter$ indicates, obviously, N-PSAI($tol$) is often considerably
superior to N-SPAI for all the test matrices except for $\tilde A$ resulting
from {\em rajat12}. The results on the last six larger matrices are more
illustrative, where PSAI($tol$)
exhibited a considerable superiority to SPAI for regular sparse
linear systems. Particularly, for {\em dc1}, when N-SPAI is applied,
BiCGStab consumed exactly the maximum 500 iterations to
achieve the accuracy requirement, while N-PSAI($tol$) only used the maximum
200 iterations; for {\em dc3}, the preconditioner produced by N-PSAI($tol$) is
much more effective than that obtained by N-SPAI, and BiCGStab preconditioned by
N-PSAI($tol$) is seven times faster than that by N-SPAI.

When applied to the original irregular sparse (\ref{eq:Axb}) directly,
S-PSAI($tol$) shows more substantial improvements over S-SPAI.
For {\em ASIC\_100k}, S-PSAI($tol$) failed to compute $M$.
S-SPAI consumed 78 hours to construct
a sparse approximate inverse $M$ of it, but $M$ is much poorer than that obtained
by N-SPAI and BiCGStab failed to converge after 500 iterations with the actual
relative residual norm $rr=1.56\times 10^{-3}$. The reason
should be that good approximate inverses of the matrix are irregular
sparse, but some columns of $M$ are too sparse to capture enough entries of
large magnitude in the corresponding columns of $A^{-1}$.
For the first four matrices, we see from Tables~\ref{psaidata}--\ref{spaidata2}
that two $M$ for each $A$ have very comparable sparsity, but the results
clearly illustrate that S-PSAI($tol$) is
considerably more effective for preconditioning than S-SPAI for the four
matrices. In terms of $iter$, S-PSAI($tol$) is eight times, three times,
twice and nearly one and a half times as fast as S-SPAI for the four problems,
respectively, as the corresponding $iter$ indicate. So S-PSAI($tol$) results
in a more substantial acceleration of BiCGStab than S-SPAI. This justifies that
PSAI($tol$) captures a better sparsity pattern of $A^{-1}$ than
SPAI for $A$ irregular sparse and computes a more effective preconditioner.

Summarizing the above, we conclude that PSAI($tol$) itself is effective
for preconditioning no matter whether a matrix is regular sparse
or not, while SPAI may work well for regular sparse
matrices but may be ineffective when $A$ is irregular sparse.
Even for regular sparse linear systems, PSAI($tol$) can outperform
SPAI considerably for preconditioning. Taking the construction cost of
preconditioners by SPAI and PSAI($tol$) into account, to make them
computationally practical, we should
apply them to regular sparse linear systems. Therefore, for
an irregular sparse linear system, a good means is to transform it into
some regular sparse problems, so that SPAI and PSAI($tol$) are relatively
efficient for computing possibly effective sparse approximate inverses.

As a last note, we make some comments on the computational efficiency
of SPAI and PSAI($tol$). Since they are different procedures that are
derived from different principles and have different features,
the computational complexity of each of them is
quite involved, and the efficiency depends on several factors including the
pattern of $A$ itself. It appears very hard, if not impossible, to compare
their flops. Hence we cannot draw any definitive conclusion on the
efficiency comparison of SPAI and PSAI($tol$). There must be cases where
one procedure wins the other, and vice versa.

Note that the PSAI($tol$) code is experimental and written in {\sc Matlab}
language for a sequential computing environment, so the performance
(run times) of PSAI($tol$) may be far from optimized. In contrast,
the SPAI 3.2 code is well programmed in C/MPI designed for distributed
parallel computers. %These differences may make the SPAI code do the job of the
%same flops considerably more rapidly than PSAI($tol$) in
%the sequential environment.
%Our experiments give us some rough
%impressions on the practical efficiency  of SPAI and PSAI($tol$).
%Tables~\ref{spaidata}--\ref{spaidata2} indicate that for given
%reasonable parameters, N-PSAI($tol$) was at least comparable to N-SPAI
%in efficiency and consumed less computational time to construct $M$.
A parallel PSAI($tol$) code in C/MPI or Fortran is involved
and will be left as our future work. We will expect that the performance of
PSAI($tol$) is improved substantially.

\section{Conclusions}\label{conclude}

The SPAI and PSAI($tol$) procedures are quite effective for
preconditioning linear systems arising from a lot of
real-world problems. However, the situation
is rather disappointing for quite common irregular sparse linear systems.
In this case, none of SPAI and PSAI($tol$) works well generally due to the
very high cost and/or possibly excessive storage requirement of
constructing preconditioners. However, for a regular sparse linear system,
we have shown that SPAI and, especially, PSAI($tol$) are efficient to
construct possibly effective preconditioners. Motivated by this
crucial feature and exploiting the Sherman--Morrison--Woodbury
formula, we have transformed the original irregular sparse linear system
into some regular sparse ones, for which SPAI and PSAI($tol$)
are practically efficient. We have considered numerous theoretical and practical
issues, including the non-singularity and conditioning of $\tilde A$. We
have proved that $\tilde A$ is ensured to be nonsingular for a number of
important classes of matrices,
and that its conditioning is generally better than $A$ for some of them.
We have derived stopping criteria for iterative solutions of new systems,
so that the approximate solution of the original problem achieves the
prescribed accuracy. Since irregular matrices are quite common
in practice, we have extended the applicability of SPAI and PSAI($tol$)
substantially. Numerical experiments have demonstrated that our approach works
well and improves the performance of SPAI and PSAI($tol$) substantially.

Our approach may be applicable to factorized sparse approximate inverse
preconditioning procedures \cite{benzi02}. Although reorderings of $A$
may help such procedures to reduce fill-ins, enhance robustness and
improve numerical stability when constructing factorized SAI preconditioners,
this is not always so and in fact even may make things worse
for $A$ irregular sparse \cite{benzi03}. Our approach may be combined
with reorderings to construct effective factorized SAI preconditioners for
regular sparse linear systems resulting from the irregular sparse one.

%We have demonstrated that our approach is more practical and effective
%for irregular sparse linear systems. Even for $A$ very large with a large
%number of irregular sparse columns, our approach is expected to work
%effectively as the cost of constructing preconditioners
%overwhelms the cost of solving linear systems themselves provided that
%preconditioners are effective. One might notice that
%the right-hand side of (\ref{stopU2}) is
%inversely proportional to $\sqrt{s}$ and an accuracy smaller than
%$\varepsilon$ is needed for (\ref{auxi}). Nevertheless, the higher accuracy
%requirements on the $s$ linear systems do not cause any difficulty since
%they decrease quite slowly with $s$ and are no more than two or three
%orders smaller than $\varepsilon$ even for $s$ ranging from
%$10^4\sim 10^6$ (note such $s$ means that $A$ is very large or huge).

Finally, we should point out that the performance results in this paper (run times)
are for modest sized and not very large problems, for which a direct solver may
be faster than any of the iterative solvers. But the goal of our paper
is to provide a new algorithm that is efficient to construct effective
sparse approximate inverses so as to reduce the number of iterations required
substantially. This has implications for very large matrices for which direct
solvers are not feasible and for a parallel computing environment.

\bigskip

{\bf Acknowledgments}. We thank the editor Professor Davis very much
for his valuable data analysis on all the test matrices
in \cite{davis2011university} and for his suggestions.
We are very indebted to the referees for their comments and suggestions.
All of these made us improve the presentation of the paper substantially.

\end{document}